\documentclass[11pt]{amsart}

\usepackage{amssymb,amsmath,amscd,amsthm,enumerate,verbatim,color,url}
\usepackage[utf8]{inputenc}
\usepackage[mathscr]{euscript}
\usepackage{bbm}
\usepackage[left=1.3in,right=1.3in,top=1in,bottom=1in]{geometry}

\definecolor{purple}{rgb}{.5,0,1}
\definecolor{orange}{rgb}{1,.5,0}
\definecolor{pink}{rgb}{1,0,.5}
\definecolor{green}{rgb}{0,.5,0}
\definecolor{gold}{rgb}{1,.623,0}

\date{}

\newcommand{\K}{{\mathbbm K}}
\newcommand{\Hi}{{\mathcal H}}
\newcommand{\Z}{{\mathbbm Z}}
\newcommand{\R}{{\mathbbm R}}
\newcommand{\C}{{\mathbbm C}}
\newcommand{\N}{{\mathbbm N}}

\newcommand{\I}{{\mathbbm I}}

\newcommand{\D}{{\mathbbm D}}
\newcommand{\U}{{\mathbbm U}}

\newcommand{\E}{{\mathcal E}}
\newcommand{\F}{{\mathcal F}}
\renewcommand{\P}{\mathrm P}

\newcommand{\spr}{{\mathrm{spr}}}

\newcommand{\Itr}{{\mathrm{Itr}}}

\newcommand{\Id}{{\mathbbm{I}}}

\newcommand{\per}{{\mathrm{per}}}

\newcommand{\M}{{\mathcal{M}}}

\renewcommand{\P}{{\mathrm{P}}}

\newcommand{\dd}{{\mathrm{d}}}

\newcommand{\ac}{{\mathrm{ac}}}

\newcommand{\pp}{{\mathrm{pp}}}

\newcommand{\SL}{{\mathbbm{SL}}}

\newcommand{\SU}{{\mathbbm{SU}}}
\newcommand{\su}{{\mathfrak{su}}}

\newcommand{\LP}{{\mathrm{LP}}}

\newcommand{\Hd}{{\mathrm{H}}}

\newtheorem{theorem}{Theorem}[section]
\newtheorem{lemma}[theorem]{Lemma}
\newtheorem{prop}[theorem]{Proposition}
\newtheorem{coro}[theorem]{Corollary}

\newtheorem{claim}{Claim}

\theoremstyle{definition}
\newtheorem{remark}[theorem]{Remark}
\newtheorem*{remark*}{Remark}
\theoremstyle{definition}
\newtheorem*{defi}{Definition}
\theoremstyle{definition}

\theoremstyle{definition}

\allowdisplaybreaks

\numberwithin{equation}{section}

\renewcommand{\Im}{\mathrm{Im}}
\renewcommand{\Re}{\mathrm{Re}}
\newcommand{\tr}{\mathrm{tr} }

\newcommand{\Dir}{{\mathrm{D}}}

\newcommand{\Leb}{{\mathrm{Leb}}}

\newcommand{\set}[1]{\left\{#1\right\}}
\newcommand{\eqdef}{\overset{\mathrm{def}}=}

\begin{document}

\title[Singular continuous spectrum for CMV matrices]{Purely singular continuous spectrum for limit-periodic CMV operators with applications to quantum walks}

\author[J.\ Fillman]{Jake Fillman}

\address{Virginia Tech}

\email{ fillman@vt.edu}

\thanks{J.\ F.\ was supported in part by an AMS-Simons travel grant, 2016--2018}

\author[D.C.\ Ong]{Darren C.\ Ong}

\address{Xiamen University Malaysia}

\email{darrenong@xmu.edu.my}

\maketitle

\begin{abstract}
We show that a generic element of a space of limit-periodic CMV operators has zero-measure Cantor spectrum. We also prove a Craig--Simon type theorem for the density of states measure associated with a stochastic family of CMV matrices and use our construction from the first part to prove that the Craig--Simon result is optimal in general. We discuss applications of these results to a quantum walk model where the coins are arranged according to a limit-periodic sequence. The key ingredient in these results is a new formula which may be viewed as a relationship between the density of states measure of a CMV matrix and its Schur function.
\end{abstract}

\setcounter{tocdepth}{1}
\tableofcontents

\section{Introduction}
Consider a discrete one-dimensional random walk model on the integers, where each integer location is assigned a weighted coin. At each time step, the walker flips the coin at her location, and moves one integer to the left if the coin lands heads, or one integer to the right if the coin lands tails. There are several interesting questions one can ask about the long-term behavior of the walker; for example, one may ask whether the walker remains close to the starting position or goes off to infinity in either direction, given particular weightings of the coins at each location.

There has been a lot of recent interest among physicists in a quantum mechanical analogue of this classical random walk, known as a \emph{discrete-time quantum walk}. In this model, the walker (which we imagine as a quantum particle) possesses a spin (either $\uparrow$ or $\downarrow$) as well as an integer location; moreover, the walker may be in a superposition of pure states, rather than being purely localized at a particular site with a definite spin. Instead of a weighted coin at each location, we have a unitary operator (which we call the quantum coin) at each location that interacts with the particle differently depending on its spin.

The paper \cite{CGMV} ushered in a revolution in the study of these 1D quantum walks. In particular, the authors relate the time-one map of a quantum walk to a \emph{CMV matrix}, that is, a pentadiagonal unitary operator on $\ell^2(\Z)$ with a repeating $2 \times 4$ block structure of the form
\begin{equation} \label{eq:stdcmvdef}
\E
=
\E_\alpha
=
\begin{bmatrix}
\ddots & \ddots & \ddots & \ddots &&& \\
& \overline{\alpha_2}\rho_1 
& -\overline{\alpha_2}\alpha_1 
& \overline{\alpha_3} \rho_2 
& \rho_3\rho_2 &\\
& \rho_2\rho_1 
& -\rho_2\alpha_1 
& -\overline{\alpha_3}\alpha_2 
& -\rho_3\alpha_2 & \\
&& \ddots & \ddots & \ddots & \ddots
\end{bmatrix},
\end{equation}
where $\alpha_n \in \D \eqdef \{ z \in \C : |z| < 1\}$ and $\rho_n = \sqrt{1-|\alpha_n|^2}$ for all $n \in \Z$. The connection arises from a gauge transformation that changes the transition matrix of the quantum walk (the matrix that takes the wavefunction of the quantum particle from one time step to the next)  into a CMV operator. Operators of the form \eqref{eq:stdcmvdef} have been studied quite heavily, and there are many powerful tools that one may use to study their spectral and dynamical characteristics, so this connection has proved extremely fruitful in recent years. Extended CMV matrices comprise a natural playground for spectral theoretic techniques, as they are canonical unitary analogs of Schr\"odinger operators and Jacobi matrices. Moreover, CMV operators are interesting in their own right, since they naturally arise not just from 1D quantum walks \cite{CGMV,DEFHV,DFO, DFV}, but also in the classical ferromagnetic Ising model \cite{DFLY1, DMY2}, and OPUC (orthogonal polynomials on the unit circle) \cite{S1,S2}.

In particular, there has been some recent success \cite{DEFHV,DFO, DFV, Joye-Merkli} in relating the spreading of a quantum walk with spectral aspects of the associated CMV operator. Concretely, given an initial state $\psi$, one may probe long-time behavior of $\psi(\ell) = U^\ell \psi$ by examining the \emph{spectral measure} $\mu = \mu_{\psi,U}$, defined by
\[
\int_{\partial \D} z^n \, \dd\mu(z)
=
\langle \psi, U^n \psi \rangle,
\quad 
n \in \Z.
\]
The spectral measure can be separated into three parts: the \emph{absolutely continuous part}, which gives no weight to sets of Lebesgue measure zero; the \emph{pure point part}, which is supported on a countable set of points; and the \emph{singular continuous part}, which is supported on a set of Lebesgue measure zero, but which gives no weight to individual points. The absolutely continuous regime typically arises when the quantum coins are arranged in a highly structured manner, for instance if we have a periodic sequence of quantum coins. Pure point spectrum (and dynanamical localization) typically arises in the opposite setting, when the structure of the quantum coins is highly disordered, for instance if the quantum coin at each location is given by an independent, identically distributed random variable on $\U(2)$ \cite{HJS,Joye2011, Joye2012}. Singular continuous spectrum typically arises in the regime in between, for example when the quantum coins exhibit aperiodic order. As such, one rough heuristic is that the more ordered the coins, the more continuous we expect the spectral measure to be.

Through a discrete-time version of the RAGE theorem (Theorem~\ref{t:rage}) one can state how the different spectral types imply conclusions about the localization or spreading behavior of the wavefunction:

\begin{enumerate}
\item If the spectral measure is pure point, then there is a bounded region from which most of the wavepacket never leaves. In other words, the wavepacket remains \emph{localized}.

\item If the spectral measure is singular continuous, the wavepacket eventually leaves any bounded region, if we calculate in a time-averaged sense. It may or may not leave the a compact region without time-averaging. 

\item If the spectral measure is absolutely continuous, most of the wavepacket eventually leaves any bounded region. In fact, it is known that (in dimension one) we have ballistic spreading in a time-averaged sense in this case \cite{DFV}. One can prove ballistic spreading without averaging in time in some specific quantum-walk scenarios in the absolutely continuous regime \cite{AVWW,F2016CMP}, but it is unclear whether one can do this for all 1D quantum walks with purely a.c.\ spectrum.

\end{enumerate}

The second setting gives us the most interesting behavior. For singular continuous spectrum \cite{DFO} shows that we have the possibility of \emph{anomalous transport}, that is where the wavepacket leaves any bounded region but at sub-ballistic speed. This approach actually provided the first explicit demonstration of anomlous transport in a quantum walk with temporally homogeneous coins; for a nonlinear, temporally inhomogeneous example, see \cite{SWH}. To understand better the spreading of the quantum walk in this situation, it will thus be useful to collect more examples of CMV operators where we also have singular continuous spectrum. 

In this paper we produce some examples in the form of limit-periodic CMV operators. We say a sequence is \emph{limit-periodic} if it is a uniform limit of periodic sequences. For example, 
\[
a(m)
=
\sum_{n=1}^\infty
\frac{1}{2^n} \cos\left(\frac{2\pi m}{n!} \right)
\]
is a limit-periodic function of $m$. Our main result is a  proof that for generic limit-periodic CMV operators, the spectrum of the operator is purely singular continuous, supported on a Cantor set of zero Lebesgue measure.

This has direct implications on understanding quantum walks whose coins are arranged according to a such a limit-periodic pattern. Almost-periodic quantum walks have received substantial interest both from the perspective of mathematics~\cite{FOZ} and of physics \cite{RMM,XQTS}.

It is also demonstrated in \cite{DFO} and \cite{DFV} that understanding the fractal characteristics of an operator's singular continuous spectrum can give us more precise information about the rate of spreading of the wavepacket of the corresponding quantum walk (in particular, the two papers provide explicit lower bounds in terms of the Hausdorff dimension of the spectral measure and explicit upper bounds in terms of an associated transfer matrix cocycle). To this end, we include a result, Theorem~\ref{t:zhdim} about the Hausdorff dimension of a limit-periodic CMV operator's spectrum. We show that a dense set of limit-periodic operators have spectrum of Hausdorff dimension zero. Unfortunately, the formulae in \cite{DFO, DFV} do not give us any useful bounds when the spectrum has Hausdorff dimension zero beyond what we already know from the RAGE theorem. However, we see this result as a possible starting point for future research, since discovering limit-periodic CMV operators with suitably continuous spectral measure will indeed produce examples with nice explicit bounds on spreading.

In order to prove the previous theorem, we had to develop new tools concerning the density of states measure for periodic CMV operators. Roughly speaking, the density of states measure is formed by considering the operator $\mathcal E$ restricted to length-$n$ subintervals of $\ell^2(\Z)$. We can assign a weight of $1/n$ to each eigenvalue of the restricted operator to obtain a normalized eigenvalue-counting measure. The density of states (henceforth: DOS) is then the weak$^*$ limit of such measures as $n$ tends to infinity (whenever such a limit exists), and it is an important object from the point of view of statistical mechanics.

We will also include a few results about the density of states measure that have applications beyond that of quantum walks, in the spirit of \cite[Section~2]{AviDam08} and \cite{DeiSim83}. In Theorem~\ref{t:Schur} we demonstrate an inequality relating the DOS with the Schur function of a periodic CMV operator. The Schur function is one of the many CMV analogues of the Weyl-Titchmarsh $m$-function for Jacobi operators. Each CMV matrix on $\ell^2(\N)$ corresponds to a Schur function, an analytic function from the open unit disk to itself. For further information, see \cite[Section~1.3]{S1} for a discussion about the Schur function of a CMV operator. 

This connection is useful because it bridges two different perspectives on the CMV operator:\ the DOS, which emerges from dynamical systems and orthogonal polynomials; and the Schur function, which is central when one instead takes a more complex-analytic approach to the CMV operator. One immediate application of this useful connection is our Theorem~\ref{t:zm}, although the Schur function only appears implicitly.

Furthermore, this theorem is especially interesting because there are actually several different CMV analogues to the Weyl-Titchmarsh $m$-function for Jacobi operators. See \cite{simanalogs} for a discussion of five of these analogues. Because of this, the complex-analytic $m$-function approach is a lot more difficult for the CMV operator compared to the Jacobi operator, and thus having a new tool is especially useful.  In fact, in the course of our proof of Theorem~\ref{t:Schur}, we see that the Schur function plays yet another role played by the $m$-function, namely, by supplying the invariant section of the M\"obius action of the transfer matrix cocycle (though this is already in some sense implicit in, e.g.\ \cite[Equation~(5.14)]{simanalogs}). 

To the best of our knowledge, this connection between the DOS and the  Schur function has not been elucidated for periodic CMV matrices, and hence, it should be of interest independently of the remainder of the paper, particularly among communities who encounter CMV matrices through avenues other than spectral theory (e.g.\ quantum walks, Ising model, and orthogonal polynomials). Concretely, variants of this connection are highly useful in the study of Schr\"odinger operators and Jacobi matrices. 

Our paper is structured as follows. In the first portion, we describe the aforementioned connection between the DOS of a periodic CMV matrix and its Schur function. In the second part of the paper, we describe a few applications of this connection to the spectral theory of CMV matrices with limit-periodic coefficients. Concretely, we show that purely singular continuous zero-measure Cantor spectrum is generic in a suitable class of limit-periodic CMV matrices and we prove optimality of a Craig--Simon-type result for the modulus of continuity of the DOS of a CMV matrix.

\section*{Acknowledgements}
We are grateful to Fritz Gesztesy, Andy Putman, and Maxim Zinchenko for helpful conversations. J.F.\ is grateful to the Simons Center for Geometry and Physics for hospitality during the program ``Between Dynamics and Spectral Theory'', during which portions of this work were completed. D.O.\ is grateful to Christian Remling for a helpful answer \cite{Remling} to a relevant question on MathOverflow.

\section{Results} \label{sec:results}
\subsection{CMV Operators with Phase Factors}
Let us now describe our results more precisely. For some of the applications that we have in mind, the structure of CMV matrices is a little restrictive, so we will work in a slightly more flexible class of operators in which we incorporate a phase factor into the usual CMV construction. More precisely, given $\alpha \in \D \eqdef \set{z \in \C : |z| < 1}$, and $\lambda \in \partial \D$, we define a $2 \times  2$ unitary matrix by
\begin{equation}\label{eq:Theta.defn}
\Theta(\alpha,\lambda)
=
\lambda
\begin{bmatrix}
\overline \alpha & \rho	\\
\rho & -\alpha
\end{bmatrix},
\quad
\rho = \rho_\alpha = \sqrt{1-|\alpha|^2}.
\end{equation}
Suppose given a sequence $\set{(\alpha_n,\lambda_n)}_{n \in \Z} \in (\D \times \partial \D)^{\Z}$. Define unitary operators on $\ell^2(\Z)$ by
\begin{equation}\label{Theta:eq}
\mathcal L
=
\bigoplus_{j \in \Z} \Theta(\alpha_{2j},\lambda_{2j}),
\quad
\mathcal M
=
\bigoplus_{j \in \Z} \Theta(\alpha_{2j+1},\lambda_{2j+1}),
\end{equation}
where $\Theta(\alpha_k,\lambda_k)$ acts on $\ell^2(\{k,k+1\})$. The corresponding \emph{phased CMV matrix} (\emph{pCMV matrix}) is defined by $\E = \E_{\alpha,\lambda} = \mathcal L\M$. It is straightforward to check that $\E$ enjoys the following matrix representation with respect to the standard basis of $\ell^2(\Z)$:
\begin{equation} \label{eq:pcmv:def}
\E
=
\begin{bmatrix}
\ddots & \ddots & \ddots & \ddots &&& \\
& \lambda_2 \lambda_1 \overline{\alpha_2}\rho_1 
& -\lambda_2 \lambda_1 \overline{\alpha_2}\alpha_1 
& \lambda_3 \lambda_2\overline{\alpha_3} \rho_2 
& \lambda_3 \lambda_2 \rho_3\rho_2 &\\
& \lambda_2 \lambda_1 \rho_2\rho_1 
& -\lambda_2 \lambda_1 \rho_2\alpha_1 
& -\lambda_3 \lambda_2 \overline{\alpha_3}\alpha_2 
& -\lambda_3 \lambda_2 \rho_3\alpha_2 & \\
&& \ddots & \ddots & \ddots & \ddots
\end{bmatrix},
\end{equation}
where the terms of the form $-\lambda_n\lambda_{n-1} \overline{\alpha_n}\alpha_{n-1}$ comprise the main diagonal, i.e.,
\[
\langle \delta_n, \E \delta_n \rangle 
= 
-\lambda_n\lambda_{n-1} \overline{\alpha_n}\alpha_{n-1} 
\text{ for } n \in \Z.
\] 
Of course, ordinary CMV matrices are a special case of these operators, obtained by setting $\lambda_n \equiv 1$. In fact, each operator of the form \eqref{eq:pcmv:def} is unitarily equivalent to a standard CMV matrix. Concretely, define a sequence $\{\gamma_n\}_{n \in \Z}$ by $\gamma_0 = \gamma_1 = 1$, and put
\[
\gamma_{2n+2} = \lambda_{2n+1}\lambda_{2n}\gamma_{2n},
\quad
\gamma_{2n+1} = \overline{\lambda_{2n}\lambda_{2n-1}}\gamma_{2n-1}.
\]
Then one can readily verify that
\begin{equation} \label{eq:gaugedef}
\Gamma \delta_n = \gamma_n \delta_n,
\quad n \in \Z
\end{equation}
defines a unitary operator with the property that $\Gamma \E_{\alpha,\lambda} \Gamma^*$ is an ordinary CMV matrix $\E' = \E_{\alpha'}$ with coefficients
\begin{equation}\label{eq:gauge.alpha'}
\alpha_{2n-1}'
=
\lambda_{2n}\lambda_{2n-1} \gamma_{2n+1}\overline{\gamma_{2n}} \alpha_{2n-1},
\quad
\alpha_{2n}'
=
\overline{\lambda_{2n}\lambda_{2n-1}\gamma_{2n}} \gamma_{2n-1} \alpha_{2n},
\quad n \in \Z.
\end{equation}

We can also express $\alpha'$ purely in terms of $\lambda$ and $\alpha$ as follows. For integers $n>0$, we have
\[
\alpha'_{2n-1}=\overline{\lambda_{2n-1}\lambda_{2n-2}^2 \lambda_{2n-3}^2\ldots \lambda_1^2\lambda_0}\alpha_{2n-1},
\]
\[
\alpha'_{2n}=\overline{\lambda_{2n}\lambda_{2n-1}^2 \lambda_{2n-2}^2\ldots \lambda_1^2\lambda_0}\alpha_{2n}.
\]
There exist similar formulae for $\alpha_k'$ when $k\leq 0$.

In the present paper, we are most interested in the case in which $\E$ is limit-periodic.
\begin{defi}
Let us say that $\E$ is \emph{periodic} if there exists $q \in \Z_+$ (the \emph{period}) for which
\begin{equation} \label{eq:percmvdef}
\alpha_{n+q} = \alpha_n,
\;
\lambda_{n+q} = \lambda_n,
\text{ for all } n \in \Z.
\end{equation}
For each $q \in \Z_+$ and $r \in (0,1)$, denote by $\P(q,r)$ the set of all periodic pCMV matrices of period $q$ with $\|\alpha\|_\infty \leq r$, and denote by
\[
\P(r)
=
\bigcup_{q=1}^\infty \P(q,r)
\]
the space of all periodic pCMV matrices with $\|\alpha\|_\infty \leq r$. We will say that $\E$ is \emph{limit-periodic} if there exist periodic matrices $\E_n$ such that $\E_n \to \E$ in operator norm. Given $0<r<1$, denote by $\LP(r)$ the set of limit-periodic $\E$ for which $\|\alpha\|_\infty \leq r$, i.e., the closure of $\P(r)$ in the operator-norm topology; finally, let
\[
\LP
=
\bigcup_{0 < r < 1} \LP(r).
\]
Notice that $\LP(r)$ is a complete metric space with respect to the metric $d(\E,\E') = \|\E - \E'\|$ for each $r \in (0,1)$, and similarly for $\LP$.
\end{defi}

There is considerable interest in determining the spectrum and spectral type of almost-periodic Schr\"odinger operators, Jacobi matrices, and CMV matrices. In this paper, we prove that zero-measure Cantor spectrum with purely singular continuous spectral type is topologically generic in the class of limit-periodic pCMV operators.

\begin{theorem}\label{t:zm}
For Baire-generic $\E \in \LP$, $\E$ has purely singular continuous spectrum supported on a Cantor set of zero Lebesgue measure.
\end{theorem}

By suitably using the construction from the previous theorem, one can also construct limit-periodic pCMV matrices with spectra of zero Hausdorff dimension.

\begin{theorem} \label{t:zhdim}
There exists a dense set of $\E \in \LP$ such that $\E$ has purely singular continuous spectrum supported on a Cantor set of zero Hausdorff dimension.
\end{theorem}

We should mention here that Theorems~\ref{t:zm} and \ref{t:zhdim} have analogs for discrete Schr\"odinger operators \cite{Avila} and for continuous Schr\"odinger operators \cite{DFL2015}. While we follow the approach of these papers, the unitary setting provides new challenges. In particular, the connection between the DOS and the Schur function, the most crucial component of the proof requires a substantially more complicated approach in the unitary setting, due to how differently $m$-functions for the Schr\"odinger operator and CMV operator behave. 

\subsection{Quantum Walks}

Now, let us precisely describe quantum walks on $\Z$ and how they may be related to CMV matrices. The state space for such quantum walks is 
\[
\Hi \eqdef \ell^2(\Z) \otimes \C^2,
\]
where the first factor corresponds to the spatial variable, and the second factor corresponds to the internal variable, the ``spin''. The elementary tensors of the form $\delta_n^\pm \eqdef \delta_n \otimes e_\pm$ with $n \in \Z$ comprise an orthonormal basis of $\Hi$, where $\{e_+,e_-\}$ denotes the usual basis of $\C^2$:
\[
e_+
=
\begin{bmatrix}
1 \\ 0
\end{bmatrix},
\quad
e_-
=
\begin{bmatrix}
0 \\ 1
\end{bmatrix}.
\]
A time-homogeneous quantum walk scenario is completely specified by a choice of local interactions given by $2 \times 2$ unitaries, called the \emph{coins}:
\begin{equation}\label{e.timehomocoins}
Q_{n}
=
\begin{bmatrix}
q^{11}_{n} & q^{12}_{n} \\
q^{21}_{n} & q^{22}_{n}
\end{bmatrix}
\in \U(2), \quad n \in \Z.
\end{equation}
To avoid degeneracies, we will always assume that $q_n^{11}, q_n^{22} \neq 0$. The \emph{update rule} of the quantum walk, which updates the system by one time step, is given as $U = SC$, where $S$ is the conditional shift $S:\delta_n^\pm \mapsto \delta_{n \pm 1}^\pm$, and $C$ applies the coins coordinatewise:
\[
C\delta_n^+
=
q_n^{11} \delta_n^+ + q_n^{21} \delta_n^-,
\quad
C\delta_n^-
=
q_n^{12} \delta_n^+ + q_n^{22} \delta_n^-.
\]
Thus, $U$ acts as follows on pure states:
\begin{equation} \label{e.updaterule}
\delta_{n}^+
\mapsto
  q^{11}_{n} \delta_{n+1}^+ 
+ q^{21}_{n} \delta_{n-1}^-,
\quad
\delta_n^-
\mapsto
  q^{12}_{n} \delta_{n+1}^+
+ q^{22}_{n} \delta_{n-1}^- 
\end{equation}
Since $S$ and $C$ are clearly unitary, this defines a unitary operator $U$ on $\mathcal{H}$. Next, order the basis of $\mathcal{H}$ as $\varphi_{2m} = \delta_m^+$, $\varphi_{2m+1} = \delta_m^-$ for $m \in \Z$. In this ordered basis, the matrix representation of $U : \Hi \to \Hi$ is precisely
\begin{equation}\label{e.umatrixrep}
U
=
\begin{bmatrix}
\ddots & \ddots & \ddots & \ddots &&&&& \\
 & 0 & 0 & q_1^{21} & q_1^{22} & && \\
 & q_0^{11} & q_0^{12} & 0 & 0 & && \\
 &&& 0 & 0 & q_2^{21} & q_2^{22}& \\
 &&& q_1^{11} & q_1^{12} & 0 & 0 & \\
 &&&& \ddots & \ddots &  \ddots & \ddots
\end{bmatrix},
\end{equation}
which follows immediately from \eqref{e.updaterule}; compare \cite[Section~4]{CGMV}.

We can connect quantum walks to CMV matrices using the following observation. If all Verblunsky coefficients with even index vanish, the extended CMV matrix in \eqref{eq:stdcmvdef} becomes
\begin{equation}\label{e.ecmvoddzero}
\mathcal{E}
=
\begin{bmatrix}
\ddots & \ddots & \ddots & \ddots &&&&& \\
& 0 & 0 & \overline{\alpha_3}  & \rho_3 &&& \\
& \rho_1 & - \alpha_1 & 0 & 0 &  &&  \\
&&& 0 & 0 & \overline{\alpha_5}  & \rho_5 & \\
&&& \rho_3 & - \alpha_3 & 0 & 0  & \\
&&&& \ddots & \ddots &  \ddots & \ddots
\end{bmatrix}.
\end{equation}
The matrix in \eqref{e.ecmvoddzero} strongly resembles the matrix representation of $U$ in \eqref{e.umatrixrep}. Note, however, that $\rho_n > 0$ for all $n$, so \eqref{e.umatrixrep} and \eqref{e.ecmvoddzero} may not match perfectly whenever $q_n^{jj}$ fails to be real and positive. However, this can be easily resolved by a gauge transformation as in \eqref{eq:gaugedef}.

From this correpondence, as an immediate consequence of Theorem~\ref{t:zm} and Theorem~\ref{t:rage}, we obtain the following result.

\begin{coro}\label{coro:qw}
For a Baire-generic class of two-sided one-dimensional quantum walks with limit-periodic coins of the form $\Theta(\alpha,\lambda)$, the wavefunction of the particle behaves as in situation {\rm(}b{\rm)} of Theorem~\ref{t:rage}. That is, it escapes any bounded region in a time-averaged sense.
\end{coro}

\subsection{Density of States}
In addition, our construction yields some interesting consequences for the density of states measures for limit-periodic operators. Let us briefly recall how the density of states is defined. Given a CMV operator $\E$, one may consider its restriction to finite boxes with Dirichlet boundary conditions, i.e., $\E^{(n)}_\Dir = \chi_n \E \chi_n^*$, where $\chi_n: \ell^2(\Z) \to \ell^2([-n,n] \cap \Z)$ denotes the canonical projection. Then $\E^{(n)}_\Dir$ has $2n+1$ eigenvalues (counting multiplicities), so we denote by $\nu_n$ the normalized eigenvalue counting measure. If $\nu_n$ has a weak$^*$ limit as $n \to \infty$, then we denote it by $\nu$ and call it the \emph{density of states measure} (henceforth: DOS) of $\E$. 

In general, this weak limit need not exist, but it does exist in almost all models of interest -- a broad class of models for which the DOS exists are \emph{ergodic} CMV matrices (sometimes called \emph{stochastic} CMV matrices). Concretely, let $S:\Omega \to \Omega$ be a invertible bi-measurable transformation of a Borel probability space $\Omega$. Given measurable functions $f:\Omega \to \D$ and $g:\Omega \to \partial \D$, we may define a family of operators $\E(\omega) = \E_{\alpha(\omega),\lambda(\omega)}$ by
\[
\alpha_n(\omega)
=
f(S^n\omega),
\quad
\lambda_n(\omega)
=
g(S^n \omega),
\quad
n \in \Z, \; \omega \in \Omega.
\]

Then, if $\mu$ is an $S$-ergodic measure and $f$ is bounded away from $\partial \D$ in the sense that $|f| \leq C < 1$ $\mu$-almost everywhere, then the DOS of $\E(\omega)$ exists for $\mu$-a.e.\ $\omega \in \Omega$ by straightforward arguments using ergodicity. If $\Omega$ is additionally a compact metric space, $f$ and $g$ are continuous, and $S$ is a uniquely ergodic homeomorphism from $\Omega$ to itself, then the DOS of $\E(\omega)$ exists for \emph{all} $\omega \in \Omega$; this follows from the equivalence of unique ergodicity of $S$ and uniform convergence of Birkhoff averages of continuous functions to constants (see, e.g.\ \cite[Theorem~6.19]{Walters}).

For the DOS, we will prove a Craig--Simon-type result on the modulus of continuity; compare \cite{CS83CMP}. Here, and throughout the paper, we use $\Leb(A)$ to denote the one-dimensional Lebesgue measure of $A \subset \partial \D$ (which we also refer to as the arc length measure). In particular, we have $\Leb(\partial \D) = 2\pi$.

\begin{theorem} \label{t:cmv:craigsim}
The DOS of an ergodic family of extended pCMV matrices is log-H\"older continuous. That is, there exists a constant $c > 0$ for which
\[
\nu(A)
\le 
-c\left( \log \Leb(A) \right)^{-1}
\]
for every arc $A \subseteq \partial \D$ with $\Leb \, A<1/2$, where $\Leb$ denotes the arc length measure on $\partial \D$.
\end{theorem}

Our construction of limit-periodic pCMV matrices with very thin spectra allows us to show that the previous result is optimal in general.

\begin{theorem} \label{t:olhcids}
Suppose that  $h:\R_+ \to \R_+$ is any increasing function that satisfies
\begin{equation} \label{e:fasterthanlogdecay}
\lim_{\delta \to 0^+} h(\delta) \log \left( \delta^{-1} \right) = 0.
\end{equation}
Then there exists a dense set of $\E \in \LP$ so that the associated density of states $\nu$ satisfies
\begin{equation} \label{eq:olhcids:conclusion}
\limsup_{z \to z_0}
\frac{\nu([z,z_0])}{h\left( |z-z_0| \right)}
=
+ \infty
\end{equation}
for any $z_0 \in \sigma(\E)$, where $[z,z_0]$ denotes the shortest closed subarc of $\partial \D$ containing $z$ and $z_0$. In particular, for such $\E$, the DOS is not $\alpha$-H\"older continuous for any $\alpha > 0$.
\end{theorem}

Since it is critical to our construction, let us briefly comment on why we need the additional flexibility supplied by pCMV matrices for some parts of the paper. Broadly speaking, one does this so that one may ``move'' the spectrum a small amount by a uniformly small change in operator data. Concretely, if we want to rotate the spectrum of a CMV matrix $\E$ by a phase $\lambda \in \partial \D$, then one clearly has
\[
\sigma(\lambda\E)
=
\lambda\sigma(\E).
\]
However, this is problematic for ordinary CMV matrices, as $\lambda\E$ is not a genuine CMV matrix unless $\lambda = 1$. One can resolve this by conjugating $\lambda \E$ by a suitable diagonal unitary $\Lambda$ as above to obtain an honest CMV matrix $\widetilde \E = \lambda \Lambda^* \E \Lambda$ with $\sigma(\widetilde\E) = \lambda \sigma(\E)$. However, while $\|\lambda\E - \E\| = |\lambda - 1|$, one completely loses control over $\| \E - \widetilde \E\|$. Thus, our desire to incorporate phases is motivated precisely by our need to force small perturbations at the level of spectral data correspond to uniformly small perturbations in the coefficient data.
\bigskip

The structure of the paper is as follows. In Section~\ref{sec:genstuff}, we collect some general facts regarding the spectral analysis of periodic pCMV matrices. In Section~\ref{sec:dos}, we establish a useful connection between the DOS measure and a type of Weyl--Titchmarsh $m$-function of a periodic CMV matrix. This connection is then used in Section~\ref{sec:expspec} to construct periodic CMV matrices with exponentially thin spectra. We prove the Craig--Simon result (Theorem~\ref{t:cmv:craigsim}) in Section~\ref{sec:craigsim} and prove that it is optimal (Theorem~\ref{t:olhcids}) in Section~\ref{sec:olhc}. We prove Theorems~\ref{t:zm} and \ref{t:zhdim} in Section~\ref{sec:cantspec}. We discuss some applications to spectral theory for OPUC in Section~\ref{sec:OPUC}. In addition, we include a proof of a discrete-time RAGE theorem in Appendix~\ref{sec:RAGE} for the convenience of the reader.

\section{Periodic CMV Matrices with Phases} \label{sec:genstuff}

In the present section, we will briefly describe some aspects regarding the spectral analysis of pCMV matrices.

\subsection{A Gesztesy--Zinchenko Formalism}

First, we will briefly outline how to work out a transfer-matrix formalism for CMV matrices with phase factors as in \eqref{eq:pcmv:def}. The overall approach follows the constructions of and is very similar to \cite{GZ}.

Suppose $\E u = z u$ and let $v = \mathcal M u$. Then, one has $\mathcal L v = \E u = zu$, so we have
\begin{equation} \label{eq:gzeigs}
\Theta(\alpha_{2n},\lambda_{2n})
\begin{bmatrix}
v_{2n}\\v_{2n+1}
\end{bmatrix}
=z
\begin{bmatrix}
u_{2n}\\u_{2n+1}
\end{bmatrix},
\quad
\Theta(\alpha_{2n-1},\lambda_{2n-1})
\begin{bmatrix}
u_{2n-1}\\u_{2n}
\end{bmatrix}
=
\begin{bmatrix}
v_{2n-1}\\v_{2n}
\end{bmatrix}
\end{equation}
for all $n \in \Z$. Rearranging \eqref{eq:gzeigs}, one obtains
\begin{align}
v_{2n-1}
& =
\lambda_{2n-1}(\overline{\alpha_{2n-1}}u_{2n-1}+\rho_{2n-1}u_{2n}),\label{v_2n-1}\\
v_{2n}
& =
\lambda_{2n-1}(\rho_{2n-1}u_{2n-1}-\alpha_{2n-1}u_{2n}).\label{v_2n}
\end{align}
Multiplying \eqref{v_2n} by $\rho_{2n-1}$, we get
\[
\rho_{2n-1}v_{2n}
=
\lambda_{2n-1}(u_{2n-1}-\alpha_{2n-1}(\overline{\alpha_{2n-1}}u_{2n-1} + \rho_{2n-1}u_{2n})),
\]
which, by \eqref{v_2n-1} is equivalent to
\begin{equation}\label{rhov_2n}
\rho_{2n-1}v_{2n}
=
\lambda_{2n-1}u_{2n-1}-\alpha_{2n-1}v_{2n-1}.
\end{equation}
Together, \eqref{v_2n-1} and \eqref{rhov_2n} give us
\begin{equation} \label{eq:gzPrec}
\begin{bmatrix}
u_{2n}\\
v_{2n}
\end{bmatrix}
=
P(\alpha_{2n-1},\lambda_{2n-1};z)
\begin{bmatrix}
u_{2n-1}\\
v_{2n-1}
\end{bmatrix},
\end{equation}
with 
\begin{equation}\label{eq:bmatrix}
P(\alpha,\lambda;z)=\frac{1}{\rho}
\begin{bmatrix}
-\overline{\alpha}& \lambda^{-1}\\
\lambda& -\alpha
\end{bmatrix}.
\end{equation}
Similarly, we have
\begin{equation} \label{eq:gzQrec}
\begin{bmatrix}
u_{2n+1}\\
v_{2n+1}
\end{bmatrix}
=
Q(\alpha_{2n},\lambda_{2n};z)
\begin{bmatrix}
u_{2n}\\
v_{2n}
\end{bmatrix},
\end{equation}
with 
\begin{equation}\label{eq:Qmatrix}
Q(\alpha,\lambda;z)
=
\frac{1}{\rho}
\begin{bmatrix}
-\alpha& \lambda z^{-1}\\
\lambda^{-1}  z& -\overline{\alpha}
\end{bmatrix}.
\end{equation}

It is helpful to identify these transfer matrices with suitable conformal transformations of the unit disk, $\D$. In particular, for every choice of $\alpha, \beta \in \D$ and $\lambda,\mu ,z \in \partial \D$, one has
\[
P(\alpha,\lambda;z)Q(\beta,\mu;z) \in \SU(1,1)
\eqdef
\set{M \in \U(1,1) : \det M = 1},
\]
where
\[
 \U(1,1)
\eqdef
\set{M \in \C^{2\times 2} : M^* \begin{bmatrix}
1 & 0 \\ 0 & -1
\end{bmatrix} M = \begin{bmatrix}
1 & 0 \\ 0 & -1
\end{bmatrix}}.
\]
This is helpful, as the group $\SU(1,1)$ is quite well-understood. An inspired reference is \cite[Section~10.4]{S2}. For example, from \cite[(10.4.16)]{S2} we know that a matrix is in $\SU(1,1)$ if and only if it can be written in the form
\[
A(p,q)
\eqdef
\begin{bmatrix}
p &q\\
\overline q & \overline p
\end{bmatrix}
\]
with $p,q\in\C$ and $\vert p\vert^2-\vert q\vert^2=1$. Additionally, each $A \in \U(1,1)$ acts on $\D$ by M\"obius transformations, namely:
\begin{equation} \label{eq:mobiusdef}
\begin{bmatrix}
a & b \\ c & d
\end{bmatrix}
\cdot z
=
\frac{az + b}{cz+d}.
\end{equation}

Now, define $Y(n;z) = P(\alpha_n,\lambda_n;z)$ whenever $n$ is odd and $Y(n;z) = Q(\alpha_n,\lambda_n;z)$ whenever $n$ is even, and put
\[
Z(n,m;z)
\eqdef
\begin{cases}
Y(n-1;z) \cdots Y(m;z) & n > m \\
\Id & n=m \\
Z(m,n;z)^{-1} & n < m.
\end{cases}
\]
In view of \eqref{eq:gzPrec} and \eqref{eq:gzQrec}, this definition yields
\[
\begin{bmatrix}
u_n \\ v_n
\end{bmatrix}
=
Z(n,m;z)
\begin{bmatrix}
u_m \\ v_m
\end{bmatrix}
\text{ for every }
n,m \in \Z.
\]

\subsection{Spectrum, Lyapunov Exponent, and DOS of Periodic pCMV Matrices}

In the present subsection, we will briefly describe a few aspects of the spectral analysis of periodic pCMV matrices. Although the propositions that we quote are generally only proved for ordinary CMV matrices in the literature, the arguments carry over \textit{mutatis mutandis} for pCMV matrices, so we do not supply detailed proofs.

\bigskip

Suppose $\E$ is $q$-periodic and $q$ is even (we do not require a minimal period, so this is no loss).  The associated \emph{monodromy matrix} is
\[
\Phi(z)
\eqdef
Z(q,0;z)
\]
and the \emph{discriminant} is defined by
\[
\Delta(z)
\eqdef
\tr \,\Phi(z).
\] 
It is straightforward to identify the spectrum of $\E$ in terms of the Lyapunov behavior of $\Phi$.

\begin{prop} \label{p:floq}
If $\E$ is $q$-periodic, the limit
\[
L(z)
=
L(z,\E)
\eqdef
\lim_{n \to \infty} \frac{1}{n}\log \|Z(n,0;z)\|
\]
exists for all $z \in \C \setminus\set{0}$, and it satisfies
\begin{equation} \label{eq:LE:specrad}
L(z)
=
\frac{1}{q} \log \mathrm{spr} \, \Phi(z),
\end{equation}
where $\mathrm{spr} \, A$ denotes the spectral radius of $A$. Moreover, the spectrum of $\E$ is given by
\begin{equation} \label{eq:Espec:LEandDiscriminant}
\sigma(\E)
=
\set{z : L(z) = 0}
=
\set{z : \Delta(z) \in [-2,2]}.
\end{equation}
\end{prop}

\begin{proof}
For standard CMV matrices, this is contained in \cite[Theorems~11.1.1 and 11.1.2]{S2}. We supply a brief sketch of an argument that proves the indicated statements; notice that, since we prove fewer statements than in \cite{S2}, we follow a somewhat different path. First, existence of the limit defining $L$ and the equality in \eqref{eq:LE:specrad} are immediate from periodicity and Gelfand's formula for the spectral radius. Since $|\det \Phi| = 1$, the second equality in \eqref{eq:Espec:LEandDiscriminant} follows from \eqref{eq:LE:specrad}. 

Next, if $\Delta(z) \in [-2,2]$, then $\Phi(z)$ has at least one unimodular eigenvalue, which may be used to construct a nontrivial bounded solution $\psi$ to $\E \psi = z\psi$ (in fact, \emph{all} solutions are bounded whenever $\Delta(z) \in (-2,2)$). Normalized cutoffs of $\psi$ produce a Weyl sequence for $\E$ at spectral parameter $z$, whence $z \in \sigma(\E)$; cf.\ \cite{DFLY2}. Conversely, if $\Delta(z) \notin [-2,2]$, then $\Phi(z)$ has eigenvalues $\lambda_\pm$ with $|\lambda_+| < 1 < |\lambda_-|$. One may use the associated eigenvectors to construct solutions $u_\pm$ to $\E u = zu$ that decay exponentially quickly at $\pm \infty$. One may use $u_\pm$ to construct $(\E-z \cdot \Id)^{-1}$, and thus $z \notin \sigma(\E)$.
\end{proof}

The function $L$ from Proposition~\ref{p:floq} is called the \emph{Lyapunov exponent}.
\medskip

Next, we relate the DOS of $\E$ to a ``rotation number.'' Concretely, for $z = e^{i\tau} \in \sigma(\E)$, one may choose $\theta = \theta(\tau)$ so that 
\[
\Delta(z)
=
2\cos\theta
\text{ and }
0 \le \theta \le \pi.
\]

\begin{theorem} \label{t:ids:rot}
If $\E$ is $q$-periodic, then $\nu$, the DOS of $\E$, exists and is absolutely continuous with respect to normalized Lebesgue measure on $\partial \D$. Moreover, on the interior of each band of $\sigma(\E)$, one has
\begin{equation}
\frac{\dd\nu}{\dd\tau}
=
\frac{1}{\pi q} \left| \frac{\dd\theta}{\dd\tau}\right|.
\end{equation}
\end{theorem}

\begin{proof}
This is essentially \cite[Equation~(11.1.28)]{S2}; we supply a sketch that follows the rough outline of the proof of \cite[Theorem~5.4.5]{simszego}.

Let $\E^{\per}_{nq}$ denote the restriction of $\E$ to $\ell^2([0,nq)\cap\Z)$ with periodic boundary conditions and denote the corresponding eigenvalues by $z_{n,1}^{\per} , \cdots , z_{n,nq}^{\per}$ listed according to multiplicity. Then, the probability measures
$$
\nu_{nq}^{\per}
=
\frac{1}{nq} \sum_{j=1}^{nq} \delta_{z_{n,j}^{\per}}
$$
defined by uniformly distributing point masses on the eigenvalues of $\E_{nq}^{\per}$ converge weakly to $\nu$ as $n \to \infty$. Moreover, one can easily check that $w$ is an eigenvalue of $\E_{nq}^{\per}$ if and only if 1 is an eigenvalue of $Z(nq,0;w) = \Phi(w)^n$, which holds if and only if the monodromy matrix $\Phi(w)$ has an $n$th root of unity as an eigenvalue. Thus, $w$ is an eigenvalue of $\E_{nq}^{\per}$ if and only if
$$
\Delta(w) 
= 
2 \cos \left( \frac{2 \pi j}{n} \right)
$$
for some integer $ 0 \leq j \leq n/2$. Notice that any such eigenvalue with $0 < j < n/2$ is necessarily of multiplicity two, while the eigenvalues corresponding to $j=0$ and $j=n/2$ (for even $n$) can have multiplicity one or two.

Let $B$ denote any band of $\sigma(\E)$, and $\Delta_B = \Delta|_B$ the restriction of $\Delta$ to $B$ so that $\Delta_B^{-1}$ is a continuous map $[-2,2] \to B$. Given a continuous compactly supported function $g$, the previous discussion implies
$$
\int_{B} g(z) \, \dd\nu_{nq}^{\per}(z)
 =
\frac{2}{nq} \left( \sum_{j=0}^{\lfloor n/2 \rfloor} g \left( \Delta_B^{-1} \left( 2 \cos\left( \frac{2 \pi j}{n} \right) \right) \right) + O(1) \right).
$$
The $O(1)$ term is bounded in absolute value by $2\|g\|_\infty$ and arises from the fact that the extremal terms may have multiplicity one. Sending $n \to \infty$, we obtain
$$
\int_{B} g(z) \, \dd\nu(z) 
= 
\frac{1}{\pi q} \int_0^{\pi} g \left( \Delta_B^{-1} ( 2 \cos(\theta) ) \right) \, \dd\theta.
$$
The conclusion of the theorem follows from the change of variables theorem.
\end{proof}

Theorem~\ref{t:ids:rot} has a useful consequence for the DOS of periodic pCMV operators, which we note here for future use; compare \cite[Theorem~11.1.3(3)]{S2}.

\begin{prop} \label{p:nu=1/q}
Let $B$ be any band of $\sigma(\mathcal E)$, where $\mathcal E$ is a $q$-periodic pCMV operator. Then $\nu(B)=\frac{1}{q}$.
\end{prop}

\begin{proof}
This follows from Theorem~\ref{t:ids:rot} and the fact that $\theta$ varies between $0$ and $\pi$ in a monotone fashion on each band of $\sigma(\E)$.
\end{proof}

\section{A Formula for the DOS} \label{sec:dos}

On the interior of a band of the spectrum of a periodic pCMV matrix, the monodromy matrices are elliptic, hence conjugate to rotations. The goal of the present section is to elucidate a connection between the the DOS and the $\SU(1,1)$ matrices that conjugate the monodromy matrices to rotations.

Concretely, suppose $\E$ is $q$-periodic (with $q$ even), and let us define for even $j$
\begin{equation} \label{eq:Tjdef}
T_j
=
T_j(\zeta)
\eqdef
Z(j+2,j;\zeta)
=
P(\alpha_{j+1},\lambda_{j+1};\zeta) Q(\alpha_j,\lambda_j;\zeta).
\end{equation}
For most of what follows, we will suppress the dependence on $\zeta$ to simplify the exposition. Put $A_0 = \Id$, and, for even $k \ge 2$, let
\begin{equation} \label{eq:Akdef}
A_k
=
T_{k-2} T_{k-4} \cdots T_2 T_0.
\end{equation}
Finally, put
\begin{equation} \label{eq:Phikdef}
\Phi_k
=
A_{k} A_q A_{k}^{-1}
=
T_{k-2} T_{k-4} \cdots T_2 T_0 T_{q-2} \cdots T_{k}.
\end{equation}
The diagonal subgroup of $\SU(1,1)$ will play a key role:
\[
\K
\eqdef
\set{R_\theta \eqdef \begin{bmatrix}
e^{i\theta} & 0 \\ 0 & e^{-i\theta}
\end{bmatrix} : \theta \in [0,2\pi)}.
\]
Notice that $R_\theta$ acts on the unit disk by a counterclockwise rotation of $2\theta$ radians. For this reason, we will call members of $\K$ \emph{rotation matrices}. We can easily characterize rotation matrices as the stabilizer of zero in $\SU(1,1)$, that is:
\begin{equation} \label{eq:K:equals:stab0}
\K
=
\set{A \in \SU(1,1) : A \cdot 0 = 0},
\end{equation}
where $A \cdot z$ denotes the M\"obius action of $A$ as in \eqref{eq:mobiusdef}. Our first goal is to discuss the conjugation of elliptic $\SU(1,1)$ matrices to rotations. This is well-known; we describe it both for the reader's convenience and to help elucidate the M\"obius-action perspective that will be helpful to us.

\begin{lemma}[\cite{S2}, Theorem~10.4.7(i)] \label{l:ellipticsu11}
Let $A\in\SU(1,1)$ with $\tr A \in (-2,2)$. Then there exists $M\in \SU(1,1)$ such that
\[
A
=
M R_\theta M^{-1}
\]
for some $\theta$. Any such $M$ is unique modulo right-multiplication by an element of $\K$.
\end{lemma}

\begin{proof}
If $A \in \SU(1,1)$ has $\tr(A) \in (-2,2)$, then the M\"obius action of $A$ on $\D$ has a unique fixed point, which can be seen from straightforward algebra (see, e.g.\ \cite[Theorem~7.4.4]{simontextbook2A}). We will call this fixed point $\xi$ and  write $A \cdot \xi = \xi \in \D$.  One can easily choose an element of $\SU(1,1)$ that moves zero to $\xi$, say
\begin{equation}\label{eq:Mdefxi}
M
=
M_\xi
\eqdef
\frac{1}{\sqrt{1-|\xi|^2}} \begin{bmatrix}
1 & \xi \\
\overline{\xi} & 1
\end{bmatrix}.
\end{equation}
Thus, $M^{-1} AM$ fixes zero, so $M^{-1} AM \in \K$ as noted above in \eqref{eq:K:equals:stab0}. Naturally, such a conjugacy is unique modulo right-multiplication by an element of the stabilizer of zero, which is exactly $\K$, as noted in \eqref{eq:K:equals:stab0}.
\end{proof}

For $z$ in the interior of a band, we have that the corresponding monodromy matrices $\Phi_j$, $j = 0,2,4,\ldots,q-2$ have trace in $(-2,2)$ and are therefore elliptic. Thus the previous theorem applies, and each $\Phi_j$ can be conjugated to an element of $\K$; for each $j$, we denote by $M_j$ an element of $\SU(1,1)$ with 
\begin{equation} \label{eq:Mjdef}
M_j^{-1} \Phi_j M_j \in \K.
\end{equation}

Our main result is the following lower bound for the derivative of the DOS, inspired by the general tools used to analyze Schr\"odinger operators with absolutely continuous spectrum; cf.\ \cite{AviDam08} and \cite{DeiSim83}.

\begin{theorem} \label{t:ids:HSnorm}
With setup as above, the derivative of the DOS satisfies
\[
\frac{\dd\nu}{\dd\tau}
\ge 
\frac{1}{4\pi q} \sum_{\ell=0}^{(q-2)/2} \|M_{2\ell}\|_2^2
\]
on the interior of each band of $\sigma(\E)$.
\end{theorem}

We will need to introduce a variant of the trace, which we will call the iTrace.
\[
\mathrm{Itr}
\begin{bmatrix}
a&b\\
c&d
\end{bmatrix}
=
\frac{a-d}{2i}.
\]
The iTrace is clearly linear, but not cyclic. However it is easy to check that we have the following weakened variant of cyclicity:
\begin{equation} \label{eq:itr:cyclic}
\mathrm{Itr}(R^{-1}AR)=\mathrm{Itr}(A)
\text{ for all } A \in \C^{2\times 2}, \; R \in \K.
\end{equation}
The motivation behind this definition is that if $\theta$ is a smooth function of $\tau \in \R$, then
\begin{equation} \label{eq:itr:motivation}
\Itr\left(R_\theta^{-1} \partial R_\theta\right)
=
\partial \theta.
\end{equation}
The iTrace arises here as a cousin of the Pfaffian, which is defined for skew-symmetric matrices by $\mathrm{Pf}(A)^2 = \det(A)$. For $2\times 2$ skew-symmetric matrices, the Pfaffian takes a particularly simple form:
\begin{equation}\label{eq:pfaffdef}
\mathrm{Pf}(A) = \frac{1}{2}( a_{21} - a_{12});
\end{equation}
cf.\ \cite{Aslaksen}. Of course, \eqref{eq:pfaffdef} makes sense for any $2\times 2$ matrix, and so may be viewed as defining a functional on, e.g.\ $\SL(2,\R)$. The iTrace is precisely the functional that arises by pushing this functional forward from $\SL(2,\R)$ to $\SU(1,1)$ via the standard unitary equivalence between the two.

Here, and in what follows, $\partial$ denotes differentiation with respect to $\tau$. Our goal is to establish an analogue of \eqref{eq:itr:motivation} for general one-parameter families of elliptic $\SU(1,1)$ matrices. That is, if $A(\tau)$ is a smooth family of elliptic $\SU(1,1)$ matrices with $M^{-1}AM = R_\theta$, we will find a simple formula for $\partial \theta$ in terms of the iTrace of the consituent matrices.
\medskip

\begin{theorem} \label{t:dtheta:itr}
Suppose $\tau \mapsto A(\tau)$ is a smooth map $\R \to \SU(1,1)$ such that $\tr(A(\tau)) \in (-2,2)$ for every $\tau$. Then, there exist smooth choices of $M = M(\tau)$ and $\theta = \theta(\tau)$ such that
\begin{equation} \label{eq:MconjAR}
M^{-1} A M
\equiv
R_\theta.
\end{equation}
Moreover,
\begin{equation} \label{eq:dthetadt:itr}
\partial\theta
=
\Itr\left( M^{-1} A^{-1} \partial A M \right).
\end{equation}
\end{theorem}

\begin{proof}
One can readily check that $\xi = \xi(\tau)$, the fixed point of $A(\tau)$ in $\D$ varies smoothly with $\tau$. As noted above,  we may choose
\[
M(\tau)
=
\frac{1}{\sqrt{1-|\xi(\tau)|^2}}
\begin{bmatrix}
1 & \xi(\tau) \\ \overline{\xi(\tau)}& 1
\end{bmatrix},
\]
and then we have \eqref{eq:MconjAR}. Differentiating \eqref{eq:MconjAR} with respect to $\tau$, we find
\[
(\partial M^{-1}) AM + M^{-1} \partial A M + M^{-1}A\partial M
=
iR
\begin{bmatrix}
\partial\theta & 0 \\
0 & -\partial\theta
\end{bmatrix}.
\]
Now, left-multiply both sides by $R^{-1}$ and apply \eqref{eq:MconjAR} to simplify:
\begin{equation} \label{eq:conjrel:deriv}
R^{-1} (\partial M^{-1})MR + M^{-1} A^{-1} \partial A M + M^{-1} \partial M
=
i
\begin{bmatrix}
\partial\theta & 0 \\
0 & -\partial\theta
\end{bmatrix}.
\end{equation}
By linearity and cyclicity of $\Itr$ (as in \eqref{eq:itr:cyclic}), we have
\[
\Itr\left(R^{-1} (\partial M^{-1}) MR+ M^{-1} \partial M\right)
=
\Itr\left((\partial M^{-1})M+ M^{-1} \partial M \right)
=
\Itr\left(\partial (M^{-1} M) \right)
=
0.
\]
Thus, \eqref{eq:dthetadt:itr} follows from \eqref{eq:conjrel:deriv} by taking the iTrace of both sides.
\end{proof}

Before we prove the main theorem, we will record a couple of calculations. Firstly, in light of Theorem~\ref{t:dtheta:itr}, we are interested in quantities of the form $\Itr(M^{-1} A M)$, where $A = T_j^{-1} \partial T_j$. Since $T_j \in \SU(1,1)$, it follows that $A \in \su(1,1)$, the Lie algebra of $\SU(1,1)$, which we may parameterize as follows:
\begin{equation} \label{eq:su11char}
\su(1,1)
=
\set{\begin{bmatrix}
iw & z \\ \overline z & -iw 
\end{bmatrix}
:
w \in \R, \; z \in \C}.
\end{equation}

\begin{prop} \label{p:itrMinvAM:calc}
For all $a,b,c,d \in \C$ and $\xi \in \D$, one has
\begin{equation} \label{eq:p:itrMinvAM:gen}
\Itr\left(M_\xi^{-1} \begin{bmatrix}
a & b \\ c & d
\end{bmatrix}
M_\xi\right)
=
\frac{1}{2i(1-|\xi|^2)}\left((1+|\xi|^2)(a-d) + 2b \overline{\xi} - 2c\xi \right).
\end{equation}
Consequently, for
\[
A
=
\begin{bmatrix}
iw & z \\ \overline z & -iw
\end{bmatrix}
\in \su(1,1),
\]
we have
\begin{equation} \label{eq:p:itrMinvAM:su11}
\Itr\left(M_\xi^{-1} A
M_\xi\right)
=
\frac{1}{1-|\xi|^2}\left((1+|\xi|^2)w + 2 \Im(z\overline{\xi}) \right).
\end{equation}
\end{prop}

\begin{proof}
One has
\begin{align*}
M_\xi^{-1} \begin{bmatrix}
a & b \\ c & d
\end{bmatrix}
M_\xi
& =
\frac{1}{1-|\xi|^2} \begin{bmatrix}
1 & -\xi \\
-\overline{\xi} & 1
\end{bmatrix}
\begin{bmatrix}
a & b \\ c & d
\end{bmatrix}
\begin{bmatrix}
1 & \xi \\
\overline{\xi} & 1
\end{bmatrix} \\
& =
\frac{1}{1-|\xi|^2} \begin{bmatrix}
1 & -\xi \\
-\overline{\xi} & 1
\end{bmatrix}
\begin{bmatrix}
a + b\overline \xi & a\xi + b \\ c+d\overline\xi & c\xi + d
\end{bmatrix} \\
& =
\frac{1}{1-|\xi|^2}
\begin{bmatrix}
a+b\overline\xi-c\xi-d|\xi|^2 & * \\
\ast & -a|\xi|^2-b\overline\xi + c\xi+d
\end{bmatrix}.
\end{align*}
Thus, \eqref{eq:p:itrMinvAM:gen} holds. Of course, \eqref{eq:p:itrMinvAM:su11} follows immediately from \eqref{eq:p:itrMinvAM:gen}.
\end{proof}

Now, we want to apply this to $A = T^{-1} \partial T$ with $T = PQ$ (a product of GZ matrices), so we now compute $T$ and $T^{-1}\partial T$ explicitly.

\begin{prop} \label{p:phiinvpartialphi:calc}
Let $\alpha,\beta \in \D$ and $\lambda,\mu \in \partial \D$ be given, and define
\[
T
=
T(\zeta)
=
P(\alpha,\lambda;\zeta) Q(\beta,\mu;\zeta).
\]
Let $\zeta = e^{i\tau}$ and denote by $\partial$ differentiation with respect to $\tau$ as usual. We have
\[
T
=
\frac{1}{\rho_\alpha\rho_\beta}
\begin{bmatrix}
\overline\alpha \beta+\frac{\zeta}{\mu\lambda} & -\frac{\overline\alpha \mu}{\zeta}-\frac{\overline{\beta}}{\lambda} \\
-\beta\lambda-\frac{\alpha\zeta}{\mu} & \frac{\mu\lambda}{\zeta}+\alpha\overline\beta
\end{bmatrix}
\]
and
\[
T^{-1} \partial T
=
\frac{i}{\rho_\beta^2}
\begin{bmatrix}
1 & - \overline\beta\mu/\zeta \\
\beta\zeta/\mu & -1
\end{bmatrix}.
\]
\end{prop}

\begin{proof}
First,
\begin{align*}
P(\alpha,\lambda;\zeta) Q(\beta,\mu;\zeta)
& =
\frac{1}{\rho_\alpha\rho_\beta}
\begin{bmatrix}
-\overline\alpha & 1/\lambda \\
\lambda & - \alpha
\end{bmatrix} 
\begin{bmatrix}
-\beta & \mu/\zeta \\
\zeta/\mu & -\overline\beta
\end{bmatrix} \\
& =
\frac{1}{\rho_\alpha\rho_\beta}
\begin{bmatrix}
\overline\alpha \beta+\frac{\zeta}{\mu\lambda} & -\frac{\overline\alpha \mu}{\zeta}-\frac{\overline{\beta}}{\lambda} \\
-\beta\lambda-\frac{\alpha\zeta}{\mu} & \frac{\mu\lambda}{\zeta}+\alpha\overline\beta
\end{bmatrix},
\end{align*}
as desired. Since $\zeta = e^{i\tau}$, we have $\partial\zeta = i \zeta$ and $\partial(1/\zeta) = -i/\zeta$. Consequently,
\[
\partial T
=
\frac{i}{\rho_\alpha\rho_\beta}
\begin{bmatrix}
\frac{\zeta}{\mu\lambda} & \frac{\overline\alpha \mu}{\zeta} \\
-\frac{\alpha\zeta}{\mu} & -\frac{\mu\lambda}{\zeta}
\end{bmatrix}.
\]
This gives
\begin{align*}
T^{-1}\partial T
& =
\frac{i}{\rho_\alpha^2\rho_\beta^2}
\begin{bmatrix}
 \frac{\mu\lambda}{\zeta}+\alpha\overline\beta
& \frac{\overline\alpha \mu}{\zeta} + \frac{\overline{\beta}}{\lambda} \\
\beta\lambda + \frac{\alpha\zeta}{\mu} 
& \overline\alpha \beta+\frac{\zeta}{\mu\lambda} 
\end{bmatrix}
\begin{bmatrix}
\frac{\zeta}{\mu\lambda} & \frac{\overline\alpha \mu}{\zeta} \\
-\frac{\alpha\zeta}{\mu} & -\frac{\mu\lambda}{\zeta}
\end{bmatrix} \\
& =
\frac{i}{\rho_\alpha^2\rho_\beta^2}
\begin{bmatrix}
1 + \frac{\alpha\overline\beta \zeta}{\mu\lambda} - |\alpha|^2 - \frac{\alpha\overline\beta \zeta}{\mu\lambda}
& \frac{\overline \alpha \mu^2\lambda}{\zeta^2} + \frac{|\alpha|^2\overline\beta \mu}{\zeta} - \frac{\overline \alpha \mu^2\lambda}{\zeta^2} - \frac{\overline\beta \mu}{\zeta}  \\
\frac{\beta\zeta}{\mu} + \frac{\alpha\zeta^2}{\mu^2\lambda} - \frac{|\alpha|^2 \beta\zeta}{\mu} - \frac{\alpha\zeta^2}{\mu^2\lambda}
& \frac{\overline\alpha \beta \mu\lambda}{\zeta} + |\alpha|^2 - \frac{\overline\alpha \beta \mu \lambda}{\zeta}-1
\end{bmatrix} \\
& =
\frac{i}{\rho_\alpha^2\rho_\beta^2}
\begin{bmatrix}
\rho_\alpha^2 & -\rho_\alpha^2 \frac{\overline \beta \mu }{\zeta} \\
\rho_\alpha^2 \frac{\beta\zeta}{\mu} & -\rho_\alpha^2
\end{bmatrix} \\
& =
\frac{i}{\rho_\beta^2}
\begin{bmatrix}
1 & - \overline\beta\mu/\zeta \\
\beta\zeta/\mu & -1
\end{bmatrix}.
\end{align*}

\end{proof}

\begin{proof}[Proof of Theorem~\ref{t:ids:HSnorm}]
Apply Proposition~\ref{p:phiinvpartialphi:calc} to $T = T_j$ to get
\[
T_j^{-1} \partial T_j
=
\frac{i}{\rho_j^2}
\begin{bmatrix}
1 & -\overline{\alpha_j} \lambda_j/\zeta \\
\alpha_j \zeta/\lambda_j & -1
\end{bmatrix}
\]
for every $j \in 2\Z$. Then, applying Proposition~\ref{p:itrMinvAM:calc} (and denoting the fixed point of $\Phi_j$ on $\D$ by $\xi_j$), we deduce that
\begin{align*}
\Itr(M_j^{-1} T_j^{-1} \partial T_j M_j)
& =
\frac{1}{\rho_j^{2}(1-|\xi_j|^2)}\left(1+|\xi_j|^2 + 2 \, \Im(-i\overline{\alpha_j} \lambda_j \zeta^{-1} \overline{\xi_j})\right) \\ 
& \ge 
\frac{1}{\rho_j^{2}(1-|\xi_j|^2)}(1+|\xi_j|^2 - 2|\alpha_j||\xi_j|).
\end{align*}
By an easy calculus exercise, we may bound the final expression from below by
\[
\frac{1}{\rho_j^{2}(1-|\xi_j|^2)}(1+|\xi_j|^2 - 2|\alpha_j||\xi_j|)
\ge
\frac{1}{\rho_j^2(1-|\xi_j|^2)}(1-|\alpha_j|^2)
= 
\frac{1}{1-|\xi_j|^2}.
\]
In light of \eqref{eq:Mdefxi}, we have
\begin{equation} \label{eq:hsnorm:fixptbd}
\frac{1}{1-|\xi_j|^2}
\ge 
\frac{1}{4} \|M_j\|_2^2.
\end{equation}
Now, one combining all of this with Theorem~\ref{t:dtheta:itr}, we get
\begin{align*}
\partial\theta
& =
\Itr(M_0^{-1} \Phi_0^{-1} \partial \Phi_0 M_0) \\
& =
\Itr\left(M_0^{-1} \left( \sum_{\ell=0}^{(q-2)/2} A_{2\ell}^{-1} T_{2\ell}^{-1} (\partial T_{2\ell}) A_{2\ell} \right) M_0\right).
\end{align*}
Since $A_{2\ell}^{-1} \Phi_{2\ell} A_{2\ell} = \Phi_0$, it follows that $A_{2\ell}M_{0}$ conjugates $\Phi_{2\ell}$ to a rotation. Then, by uniqueness (cf.\ Lemma~\ref{l:ellipticsu11}) $A_{2\ell} M_{0} =  M_{2\ell}R_{2\ell}$ for some $R_{2\ell} \in \K$. Thus, by cyclicity of the iTrace as in \eqref{eq:itr:cyclic}, we deduce
\begin{align*}
\Itr\left(M_0^{-1} \left( \sum_{\ell=0}^{(q-2)/2} A_{2\ell}^{-1} T_{2\ell}^{-1} (\partial T_{2\ell}) A_{2\ell} \right) M_0\right)
& =
\Itr\left(\sum_{\ell=0}^{(q-2)/2} M_{2\ell}^{-1} T_{2\ell}^{-1} (\partial T_{2\ell}) M_{2\ell} \right) \\
& \ge 
\frac{1}{4} \sum_{\ell=0}^{(q-2)/2} \|M_{2\ell}\|_2^2,
\end{align*}
where the inequality is from \eqref{eq:hsnorm:fixptbd}. By Theorem~\ref{t:ids:rot}, we have
\[
\partial \nu
=
\frac{1}{\pi q} \partial\theta,
\]
so the theorem is proved.
\end{proof}

\section{Periodic Operators with Exponentially Thin Spectra} \label{sec:expspec}

In the present section, we use the DOS formula from the previous section to perform our key perturbative construction. The first step is to show that we may perturb the phases to open all spectral gaps.

\begin{lemma}\label{l:opengaps}
Suppose $\alpha$ is $q$-periodic, that $q$ is even, and give $(\partial \D)^q$ the topology induced by the uniform metric
\[
\|\lambda - \lambda'\|_\infty
=
\max_j |\lambda_j - \lambda'_j|.
\]
There is a dense open set $G\subseteq (\partial \D)^q$, such that $\sigma(\E_{\alpha,\lambda})$ has $q$ connected components for every $\lambda \in G$. {\rm(}NB: We use $\lambda$ both for the element of $(\partial \D)^q$ and the $q$-periodic extension to all of $\Z${\rm)}.
\end{lemma}

\begin{remark*}
This was proved for continuum Schr\"odinger operators in \cite{S76}. The following proof is based on a different version of that argument given in \cite{Remling}.
\end{remark*}

\begin{proof}
Let $\mu$ be a $q$-periodic sequence in $(\partial \D)^\Z$. For $n$ even let $T_n(\mu)$ refer to the $2$-step transfer matrix corresponding to multiplying the $P$ and $Q$ Gesztesy-Zinchenko matrices \eqref{eq:bmatrix} and \eqref{eq:Qmatrix} corresponding to $(\alpha,\mu)$, as in \eqref{eq:Tjdef}, that is,
\[
T_n(\mu)
=
T_n(\mu;z)
\eqdef
P(\alpha_{n+1},\mu_{n+1};z) Q(\alpha_n,\mu_n;z).
\]
Throughout the argument, let us suppress the $z$-dependences of the transfer matrices to make things a little easier to read. Similarly, let $T_n(\lambda)$ refer to the $2$-step transfer matrix corresponding to the perturbed sequence $(\alpha,\lambda)$, and define the $n$-step transfer matrices $A_n(\mu)$ and $A_n(\lambda)$ as in \eqref{eq:Akdef}. As before, let $A_0(\lambda)$ be the identity matrix. Then, we have 
\[
A_{n}(\lambda)
=
T_{n-2}(\lambda) A_{n-2}(\lambda)
\]
for each $n \in 2\Z_+$. 
\medskip

Now, assume that $\sigma(\E_{\alpha,\mu})$ has a closed gap at the implicit spectral parameter, $z$. Then, $A_q(\mu)$ must be $\pm \Id$, by Floquet theory (compare \cite[Theorem~11.2.3]{S2}). Let us consider
\begin{equation}
W_n(\lambda)
\eqdef
A_n(\mu)^{-1}A_n(\lambda),
\label{eq:Y.expansion}
\end{equation}
and note that $W_0(\lambda)$ is the identity matrix. Let $\theta_j$ be the argument of $\lambda_j\overline{\mu_j} - 1$, so that
\begin{equation} \label{eq:phasechoice}
\lambda_j\overline{\mu_j} - 1
= 
\vert\mu_j - \lambda_j \vert e^{i\theta_j}.
\end{equation}
Our goal is to compute an expansion of $\tr\,W_q$ in powers of $|\lambda_j - \mu_j|$, which we will accomplish by inductively peeling off factors of the form $T_n(\mu)^{-1} T_n(\lambda)$. Concretely, first notice that
\[
P(\alpha,\mu;z)^{-1} P(\alpha,\lambda;z)
=
\frac{1}{\rho_\alpha^2}
\begin{bmatrix}
\lambda\overline{\mu} - |\alpha|^2 & \alpha(\overline\lambda - \overline\mu) \\
\overline\alpha(\lambda - \mu) & \mu\overline\lambda - |\alpha|^2
\end{bmatrix}
=
\Id + \frac{1}{\rho_\alpha^2}
\begin{bmatrix}
\lambda\overline{\mu} - 1 & \alpha(\overline\lambda - \overline\mu) \\
\overline\alpha(\lambda - \mu) & \mu\overline\lambda - 1
\end{bmatrix}.
\]
In light of this, we denote by
\[
R(\alpha,\mu,\lambda)
=
\frac{1}{1-|\alpha|^2}
\begin{bmatrix}
\lambda\overline{\mu} - 1 & \alpha(\overline\lambda - \overline\mu) \\
\overline\alpha(\lambda - \mu) & \mu\overline\lambda - 1
\end{bmatrix},
\]
and we note that \eqref{eq:phasechoice} implies that
\begin{equation} \label{eq:Ramltrace}
\tr \, R(\alpha_j,\mu_j,\lambda_j) 
= 
2 \rho_j^{-2} |\mu_j - \lambda_j| \cos\theta_j
\text{ for every } j.
\end{equation}
Now, pick any $\lambda \in (\partial \D)^q$ such that $\lambda_j = \mu_j$ for every \emph{even} $1 \le j \le q$, and denote by
\[
Q_j 
= 
Q(\alpha_j,\mu_j;z) 
= 
Q(\alpha_j,\lambda_j;z)
\]
for each even $j$. We then have
\begin{align*}
 W_{q}(\lambda)
& =
A_q(\mu)^{-1}A_q(\lambda)\\
& =
A_{q-2}(\mu)^{-1}
Q_{q-2}^{-1} 
\left(
\Id + R(\alpha_{q-1},\mu_{q-1},\lambda_{q-1})
\right)
Q_{q-2}
A_{q-2}(\lambda)\\
& =
A_{q-2}(\mu)^{-1}
\left(
\Id + Q_{q-2}^{-1} 
R(\alpha_{q-1},\mu_{q-1},\lambda_{q-1})
Q_{q-2}
\right)
A_{q-2}(\mu) W_{q-2}(\lambda)
\end{align*}
We can then apply a similar calculation to $W_{q-2}$ in the last equation. Repeating the process by using \eqref{eq:Y.expansion} with $n=q-2,q-4,\ldots$ until we reach $W_0 = \Id$, we obtain 
 \begin{align*}
W_q(\lambda) 
& =
\prod_{j \text{ even, } j=q-2}^{0}  A_{j}(\mu)^{-1}
\left(
\Id + Q_{j}^{-1} 
R(\alpha_{j+1},\mu_{j+1},\lambda_{j+1})
Q_{j}
\right)
A_{j}(\mu)\\
& =
\prod_{j \text{ even, } j=q-2}^{0}  
\left(
\Id + A_{j}(\mu)^{-1} Q_{j}^{-1} 
R(\alpha_{j+1},\mu_{j+1},\lambda_{j+1})
Q_{j}A_{j}(\mu)
\right)
 \end{align*}
By the cyclicity of the trace and \eqref{eq:Ramltrace}, this implies that
\[
\tr(W_q(\lambda))
=
2 + 2\sum_{j \text{ odd, } j=1}^{q-1} \rho_j^{-2}\vert\mu_j - \lambda_j\vert \cos(\theta_j) + O\left( \|\lambda - \mu\|_\infty^2 \right).
\]
Note that for $\lambda_j\overline{\mu_j}$ on the unit circle and close to $1$, $\cos(\theta_j)>0$. Thus for $\Vert\lambda-\mu\Vert_\infty$ sufficiently small (and nonzero), it must be true that  $ \mathrm{Tr}(W_q(\lambda)) >2$. Since $A_q(\mu) = \pm \Id$, it follows that
\[
\tr \, A_q(\lambda)
=
\tr( A_q(\mu) W_q(\lambda) )
=
\pm \tr\,W_q(\lambda)
\in \R \setminus[-2,2],
\]
and thus our implicit spectral parameter $z$ is in an open gap of $\sigma(\E_{\alpha,\lambda})$. Thus, $G$ is dense, as claimed. To see that $G$ is open in $(\partial \D)^q$, simply note that the edges of the bands of the spectrum may move by no more than $\varepsilon$ in the usual metric on $\C$ if $\mu$ is perturbed by an amount $\varepsilon$ in $\|\cdot\|_\infty$; this follows from standard eigenvalue perturbation theory.
\end{proof}

\begin{remark}
Notice that the argument still works if we only perturb some of the phases (in other words, there is a subset $S$ of $\{0,\ldots,q-1\}$ such that $\lambda_j=\mu_j$ for all $j\notin S$, and we use the topology generated by the norm $\Vert \lambda\Vert=\max_{j\in S} \vert \lambda_j\vert$). In particular, Lemma~\ref{l:opengaps} holds if we only perturb one phase \`a la \cite[Claim~3.4]{Avila}. 
\end{remark}

Using our work thus far, we may now perform our key construction: a dense set of periodic CMV operators with exponentially thin spectra.

\begin{lemma} \label{l:lp:smallspec}
Given $q \in 2\Z_+$, $r \in (0,1)$, $\E \in \P(q,r)$, and $\delta > 0$, there exist constants $c = c(\E,\delta)$, $N_0 = N_0(\E,\delta) > 0$ such that, for any $n \geq N_0$, there exists $\E' \in \P(nq,r)$ such that $\|\E - \E'\| < \delta$, and
\[
\Leb(\sigma(\E'))
\le 
e^{-cnq}.
\]
\end{lemma}

\begin{proof}
The broad strokes of the proof follow \cite{DFL2015}, which refined the techniques of \cite{Avila}. Let $\E \in \P(q,r)$ and $\delta > 0$ be given. Fix $n'$ large enough that
\begin{equation} \label{eq:nprimelarge}
\frac{4\pi}{n'q}
<
\frac{\delta}{6}
\end{equation}
By Lemma~\ref{l:opengaps}, there exists $\E' \in \P(n'q,r) \cap B_{\delta/2}(\E)$ such that $\sigma(\E')$ has exactly $n'q$ connected components. Furthermore, the Lebesgue measure of each band of $\sigma(\E')$ is at most $\frac{4\pi}{n'q}$ by Proposition \ref{p:nu=1/q} and Theorem~\ref{t:ids:HSnorm}. This implies that the Lebesgue measure of each band of $\sigma(\E')$ is less than $\delta/6$. 

\begin{claim}
There is a finite set $\F = \{\E_1',\ldots \E_\ell'\} \subseteq B_\delta(\mathcal E) \cap \P(n'q,r)$ such that
\[
\bigcap_{j=1}^\ell \sigma(\E_j')
=
\emptyset.
\]\label{claim:cover}
\end{claim}

\begin{proof}[Proof of claim]
Let $\gamma_0$ denote the minimum arc length of a gap of $\sigma(\E')$, and define
\[
\gamma
=
\min\left(\frac{\delta}{6},\frac{\gamma_0}{2}\right),
\quad
k
=
\left\lceil\frac{\delta}{6\gamma}\right\rceil.
\]
Then, we construct a family of pCMV operators $U_{-k}, \ldots, U_k$ by 
$U_j = e^{ij\gamma} \E'$. Since each band of $\sigma(\E')$ has arc length less than $\delta/6$, it is clear that the resolvent sets of these operators cover the unit circle; moreover, using the inequality
\[
|e^{i\theta} -1|
\le |\theta|
\]
for $\theta \in \R$, we see that $U_j \in B_{\delta/2}(\E') \cap \P(n'q,r)$ for each $j$. Thus, the collection $\mathcal F = \{U_j: |j| \le k\}$ fulfills the claim.
\end{proof}

By Proposition~\ref{p:floq}, we have
\begin{equation}\label{equation:positive.Lyapunov}
\eta
\eqdef
\min_z\max_{1\leq j\leq \ell}
 L(z,\E_j)
>
0.
 \end{equation}
Now, let $n \in \Z_+$ with $n > 4\ell n'$ be given, and choose $\widetilde n \in \Z_+$ maximally subject to the constraint
\begin{equation} \label{eq:expspec:tilden:cond}
\ell n'(\widetilde n + 1)
\leq 
n.
\end{equation}
We generate an $nq$-periodic CMV operator $\widetilde\E$ by concatenating the Verblunsky coefficients $\alpha'^{(j)}$ and phases $\lambda'^{(j)}$ of the $\E'_j$ operators in Claim~\ref{claim:cover}. We concatenate each pair $(\alpha'^{(j)},\lambda'^{(j)})$ at least $\widetilde n+1$ times.
 
More precisely, let us denote $s_\ell = nq$, $s_j:= j(\widetilde n+1)n'q$ for each integer $0\leq j < \ell$, and define $\widetilde \E$ to be the $nq$-periodic CMV operator with coefficients $(\widetilde \alpha, \widetilde \lambda)$ defined by
\begin{equation}\label{equation:tilde.alpha.defn}
\widetilde\alpha_m
=
\alpha^{(j)}_m ,\quad
\widetilde \lambda_m
=
\lambda^{(j)}_m,
\quad m\in[s_{j-1},s_j-1],\quad
1 \le j \le \ell.
\end{equation}
Now suppose $z$ is chosen so that at $\widetilde\Delta(z) \in (-2,2)$, where $\widetilde \Delta$ denotes the discriminant of $\widetilde\E$. By \eqref{equation:positive.Lyapunov} there is a $j\in\{1,\ldots,\ell\}$ such that $L(z, \E'_j)\geq\eta$. But the associated transfer matrices over subintervals of  $[s_{j-1},s_j-1]$ of length $\widetilde nn'q$ are exponentially large. More explicitly, denote the GZ transfer matrices associated to $\widetilde\E$ by $\widetilde Z(n,m;z)$, and those associated to $\E_j'$ by $Z_j(n,m;z)$. Then, for $k\in [0, n'q-1]$ we have 
\begin{align}
\nonumber\Vert \widetilde Z(k + s_j-n'q,k + s_{j-1};z)\Vert
& \geq
\spr\left(Z_j(k + n'q,k;z)^{\widetilde n} \right) \\
& =
\spr\left(Z_j(k + n'q,k;z)\right)^{\widetilde n} \nonumber\\
& =
\exp(\widetilde n n'q L(z, \E_j')) \nonumber\\
& \geq 
\exp(nq\eta/2\ell)\label{equation:Abound}
\end{align}
For the last step, we used that $\widetilde n$ is chosen maximal so that \eqref{eq:expspec:tilden:cond} holds, and hence
\begin{equation} \label{eq:expspec:nlargecond}
\ell n'(\widetilde n + 2) > n.
\end{equation}
Since we chose $n > 4\ell n'$, \eqref{eq:expspec:nlargecond} yields
\[
n
<
2\ell n'(\widetilde n +2) - n
<
2\ell n'(\widetilde n + 2) - 4\ell n'
=
2\ell n' \widetilde n,
\]
which gives us \eqref{equation:Abound}.

The estimate implies lower bounds on the norms of the matrices that conjugate the monodromy matrices into rotations. Specifically, let $\widetilde T$, $\widetilde \Phi$, and $\widetilde M$ be the matrices defined in \eqref{eq:Tjdef}, \eqref{eq:Phikdef}, and \eqref{eq:Mjdef} and associated with $\widetilde\E$. Then, for each even $k$, we have $X^{-1} \widetilde \Phi_{k+s_{j-1}} X \in \K$, where
\[
X
=
\widetilde M_{k+s_{j-1}}
\text{ and }
X
=
\widetilde Z(k + s_j-n'q,k + s_{j-1})^{-1} \widetilde M_{k + s_j - n'q} ,
\]
by $nq$-periodicity of $\widetilde \E$ and the definition of $\widetilde M$. By Lemma~\ref{l:ellipticsu11}, the $\widetilde M$'s are unique modulo right-multiplication by diagonal unitary elements. In other words,
\[
\widetilde Z(k + s_j-n'q,k + s_{j-1})^{-1} \widetilde M_{k + s_j - n'q}
=
\widetilde M_{k+s_{j-1}} U
\]
for some $U \in \K$. Using \eqref{equation:Abound}, we have that
\[
\max(\Vert  M_{k + s_j - n'q} \Vert, \Vert \widetilde M_{k+s_{j-1}}\Vert)
\geq
\exp(nq\eta/4\ell).
\]
Notice that this bound is uniform over $z \in \partial \D$ with $\widetilde\Delta(z) \in (-2,2)$. Thus for any band $J$ of $z\in\sigma(\widetilde\E)$, we have
\[
\Leb(J)\leq 4\pi e^{-nq\eta/2\ell}
\] 
by Proposition~\ref{p:nu=1/q} and Theorem~\ref{t:ids:HSnorm}. There are $nq$ bands in $\widetilde\E$, which then implies
\[
 \Leb(\sigma(\widetilde\E))
\leq 
4\pi nq e^{-nq\eta/2\ell}.
\] 
The desired estimate then follows for $c < \frac{\eta}{2\ell}$ and sufficiently large $n$.
\end{proof}

Using the previous theorem, we get generic zero-measure Cantor spectrum, almost for free.

\begin{theorem} \label{t:Udelt:denseopen}
Given $\delta>0$, let $M_\delta$ denote the set of limit-periodic pCMV operators whose spectra have Lebesgue measure less than $\delta$. For all $\delta>0$, $M_\delta$ is a dense, open subset of $\LP$.
\end{theorem}
\begin{proof}
By Lemma~\ref{l:lp:smallspec}, $M_\delta$ is dense, so it suffices to show that it is open. To that end, suppose that $\E\in M_\delta$. We will show that there exists an $r>0$ such that 
\[
\Leb{(\sigma(\E'))}<\delta
\]
whenever $\Vert \E-\E'\Vert<r$. Let
\[
B_\varepsilon
=
B_\varepsilon(\sigma(\E))
=
\set{z \in \partial \D : \inf_{w \in \sigma(\E)} |z-w| < \varepsilon}.
\]
We can clearly find an $\varepsilon$ so that $\Leb(B_\varepsilon)<\delta$. We simply set $r=\varepsilon$, and we are done by standard perturbative arguments.
\end{proof}

\section{Craig--Simon for CMV Matrices} \label{sec:craigsim}

In the present section, we will prove Theorem~\ref{t:cmv:craigsim}. The key fact here is known as the \emph{Thouless formula} (See for instance \cite[Equation~(10.16.5)]{S2}). One has
\begin{equation} \label{eq:thouless}
\int_{\partial \D} \log|z-w| \, \dd \nu(w)
=
L(z) + \log\rho_\infty
\text{ for all } z \in \C,
\end{equation}
where
\[
\log \rho_\infty 
=
\frac{1}{2} \int_\Omega \! \log(1-|f(\omega)|^2) \, \dd\mu(\omega),
\]
and $L(z)$ denotes the \emph{Lyapunov exponent} of the transfer matrix cocycle, defined by
\[
L(z)
=
\lim_{n \to \infty} \frac{1}{n} \int_{\Omega} \log\|Z_\omega(n,0;z) \| \, \dd\mu(\omega),
\]
where $Z_\omega$ denotes the family of transfer matrices associated to $\E(\omega)$. Since the transfer matrices satisfy $\det Z_\omega(n,0;z) = (-1)^n$, we clearly have $L(z) \ge 0$ for all $z \in \C$.

\begin{proof}[Proof of Theorem~\ref{t:cmv:craigsim}]
Let $A \subset \partial \D$ be an arc with $\Leb(A) < 1/2$, and denote $A = [z,z']$, i.e., $A$ is the shortest arc joining $z$ to $z'$. Define $A_j$, $j=1,2,3$ by $A_1 = A$,
\begin{align*}
A_2
& =
\set{w \in \partial \D : d(w,z) \leq 1, \text{ and } w \notin A_1} \\
A_3
& =
\set{w \in \partial \D : d(w,z) > 1}.
\end{align*}
By the Thouless formula, we have
\begin{align*}
0 
& \leq
L(z) \\
& =
\log\rho_\infty^{-1} + \int_{\partial \D} \log|z-w| \, \dd\nu(w) \\
& =
\log\rho_\infty^{-1} + \int_{A_1}\log|z-w| \, \dd\nu(w) + \int_{A_2}\log|z-w| \, \dd\nu(w) + \int_{A_3}\log|z-w| \, \dd\nu(w).
\end{align*}
The second integral is negative, so, upon rearrangement, we have
\[
-\int_{A_1} \log|z-w| \, \dd\nu(w)
\leq
\log\rho_\infty^{-1} + \int_{A_3} \log|z-w| \, \dd\nu(w).
\]
Since $\log|z-w| \leq \log|z-z'|$ for $w \in A_1$ and $\log|z-w| \leq \log 2$ for all $w \in \partial \D$, we obtain
\begin{equation} \label{eq:craigsim1}
- \nu(A_1) \cdot \log|z-z'|
\leq
\log 2 + \log\rho_\infty^{-1}.
\end{equation}
Moreover, $\Leb(A) \ge |z-z'|$, so we get
\begin{equation} \label{eq:craigsim2}
-\nu(A_1)\cdot \log|z-z'|
\geq 
-\nu(A_1) \log\Leb(A).
\end{equation}
Combining \eqref{eq:craigsim1} and \eqref{eq:craigsim2}, we obtain
\[
-\nu(A_1)\log\Leb(A)
\leq 
\log(2\rho_\infty^{-1}),
\]
which yields the statement of Theorem~\ref{t:cmv:craigsim} with $c =  \log(2/\rho_\infty)$.

\end{proof}

\section{Optimality of Log-H\"older Continuity of the DOS} \label{sec:olhc}

Using the construction from Section~\ref{sec:expspec}, we are able to show that the Craig--Simon theorem (Theorem~\ref{t:cmv:craigsim}) is optimal. Our approach is inspired by \cite{KruGan}; however, we note that by making Lemma~\ref{l:lp:smallspec} explicit, we are able to streamline their construction substantially.

\begin{proof}[Proof of Theorem~\ref{t:olhcids}] We will prove this theorem by constructing a sequence of periodic approximants whose spectra grow small exponentially quickly. This exponential decay is precisely the phenomenon which causes the DOS to have a poor modulus of continuity. To begin, let us write
\[
g(\delta)
=
h(\delta)\log(\delta^{-1})
\]
and note that $g(\delta) \to 0$ as $\delta \to 0$ by assumption. Let a $q$-periodic $\E_0$ (with $q$ even) and $\varepsilon_0 > 0$ be given. We will construct a sequence of periodic operators $\E_1,\E_2,\ldots$ so that $\E_\infty = \lim \E_j$ exists, $\E_\infty \in B_{\varepsilon_0}(\E_0)$, and so that the DOS of $\E_\infty$ satisfies \eqref{eq:olhcids:conclusion} at every point of its spectrum. Throughout the proof, we denote
\[
\Sigma_j
\eqdef \sigma(\E_j),\quad
0 \le j \le \infty.
\]
Take $\varepsilon_1 = \varepsilon_0/2$, and let $0 < c_1 < c(\E_0,\varepsilon_1)$, where $c(\E_0,\varepsilon_1)$ is the constant from Lemma~\ref{l:lp:smallspec}, and choose $\E_1$ to be $\varepsilon_1$-close to $\E_0$ and of period $q_1$ such that
\[
\Leb(\Sigma_1)
\leq
\frac{1}{3} e^{-c_1 q_1}.
\]

Inductively, having defined a $q_{n-1}$ periodic $\E_{n-1}$ and tolerances $\varepsilon_1,\ldots,\varepsilon_{n-1}$, let
\begin{equation} \label{eq:olhcids:epsn:choice}
\varepsilon_n
=
\min\left(\frac{\varepsilon_{n-1}}{2}, \frac{\Leb(\Sigma_{n-1})}{3} \right).
\end{equation}
By Lemma~\ref{l:lp:smallspec}, choose a $q_n$-periodic operator $\E_n$ such that $\|\E_{n-1}-\E_n\| < \varepsilon_n$ and
\begin{equation}\label{eq:olhcids:expsmallspec}
\Leb(\Sigma_n)
\leq
\frac{1}{3} e^{-c_n q_n},
\end{equation}
where $0 < c_n < c(\E_{n-1}, \varepsilon_n)$. We further insist that $q_n$ is taken large enough that
\[
g(e^{-c_n q_n})
\leq
\frac{c_n}{2n}.
\]
Evidently $\E_\infty = \lim \E_j$ exists and is $\varepsilon_0$-close to $\E_0$. Let $z_\infty \in \Sigma_\infty$ be given. For each $\ell \in \Z_+$, we may choose $z_\ell \in \Sigma_\ell$ with
\begin{equation} \label{eq:zlzinfty:dist}
|z_\infty - z_\ell| 
\leq 
\|\E_\infty - \E_\ell\|
\le
\sum_{j = \ell+1}^\infty \varepsilon_j
\le
2\varepsilon_{\ell+1}\
\le 
\frac{2}{9} e^{-c_\ell q_\ell},
\end{equation}
where the third inequality follows from \eqref{eq:olhcids:epsn:choice} and the final inequality follows from \eqref{eq:olhcids:epsn:choice} and \eqref{eq:olhcids:expsmallspec}. Suppose that the band of $\Sigma_\ell$ containing $z_\ell$ is $[\alpha_\ell,\beta_\ell] \subseteq \partial \D$ (with the band running counterclockwise from $\alpha_\ell$ to $\beta_\ell$). Then, since the DOS of $\E_\ell$ gives weight $1/q_\ell$ to $[\alpha_\ell,\beta_\ell]$ and $\|\E_\ell - \E_\infty\| < 2 \varepsilon_{\ell+1}$, there exists a choice of $z_\ell' \in \{\alpha_\ell e^{-\pi i\varepsilon_{\ell+1}},\beta_\ell e^{\pi i \varepsilon_{\ell+1}}\}$  such that
\[
\nu([z_\infty, z_\ell']) 
\geq
\frac{1}{2q_\ell}.
\]
By \eqref{eq:olhcids:expsmallspec}, \eqref{eq:zlzinfty:dist}, and our choice of $z_\ell'$ we have
\[
|z_\ell - z_\ell'|
\le 
\left( \frac{1}{3} + \frac{\pi}{9} \right) e^{-c_\ell q_\ell}.
\]
Combining this with \eqref{eq:zlzinfty:dist}, we have
\[
|z_\infty - z_\ell'|
\le 
\frac{5+\pi}{9} 
e^{-c_\ell q_\ell}
<
e^{-c_\ell q_\ell}.
\]
Consequently,
\begin{align*}
\frac{\nu([z_\infty, z_\ell'])}{h(|z_\infty - z_\ell'|)}
& \ge
\frac{1}{2q_\ell h(e^{-c_\ell q_\ell})} \\
& =
\frac{c_\ell}{2 g(e^{-c_\ell q_\ell})} \\
& \ge
\ell,
\end{align*}
so we get
\[
\limsup_{z \to z_\infty} \frac{\nu([z_\infty,z])}{h(|z_\infty - z|)}
=
\infty,
\]
as desired.
\end{proof}

\section{Limit-Periodic CMV Matrices with Zero-Measure Cantor Spectrum} \label{sec:cantspec}

\subsection{Zero Lebesgue Measure}

\begin{proof}[Proof of Theorem~\ref{t:zm}]
Let $M_\delta$ denote the set of limit-periodic CMV matrices whose spectra have measure at most $\delta$. By Theorem~\ref{t:Udelt:denseopen}, $M_\delta$ is a dense open subset of $\LP$ for every $\delta$, and hence
\[
\mathcal M
\eqdef
\set{\E \in \LP : \Leb(\sigma(\E)) = 0}
=
\bigcap_{\delta > 0} M_\delta
\]
is a dense $G_\delta$ set by the Baire category theorem.
\medskip

Since the spectrum lacks isolated points by ergodicity, it suffices to show that the limit-periodic CMV matrices with no eigenvalues contain a dense $G_\delta$ set.  However, this is clear from Gordon-type arguments. Concretely, let $\mathcal G \subseteq \LP$ denote the set of \emph{Gordon operators} in $\LP$. More specifically, $\E \in \mathcal G$ if and only if $\E \in \LP$ and there exist $q_k \to \infty$ such that
\[
\lim_{k \to \infty} C^{q_k} \max_{-q_k+1 \le n \le q_k} \big( |\alpha_n - \alpha_{n+q_k}| + |\lambda_n - \lambda_{n+q_k}| \big)
=
0
\text{ for all }  C>0.
\]
It is easy to see that $\mathcal G$ is dense in $\LP$ (as it contains all periodic operators), and that it is a $G_\delta$ set. Since $\sigma_{\mathrm{pp}}(\E) = \emptyset$ for every $\E \in \mathcal G$ \cite{F2016PAMS},\footnote{\cite{F2016PAMS} works with ordinary CMV operators, \emph{sans} phases, but it is easy to see that those arguments generalize to the present setting.} we obtain the desired result with $\mathcal C = \mathcal G \cap \mathcal M$, by the Baire Category Theorem.
\end{proof}

\subsection{Zero Hausdorff Dimension}

Here, we will prove Theorem~\ref{t:zhdim}. To establish notation, let us briefly recall the definitions of Hausdorff measures and dimension on the circle; for further details, see \cite{Falconer}.

Given $S \subseteq \partial \D$ and $\delta > 0$, a $\delta$-\emph{cover} of $S$ is a countable collection of arcs $\{I_j\}$ with the property that $S \subseteq \bigcup_j I_j$ and $\Leb(I_j) < \delta$ for each $j$. Then, for each $\gamma \geq 0$, one defines the $\gamma$-dimensional \emph{Hausdorff measure} of $S$ by
\[
h^\gamma(S)
=
\lim_{\delta \downarrow 0} \inf_{\delta\text{-covers}} \sum_j \Leb(I_j)^\gamma.
\]
For each $S \subseteq \partial \D$, there is a unique $\gamma_0 \in [0,1]$ such that
\[
h^\gamma(S)
=
\begin{cases}
\infty & 0 \leq \gamma < \gamma_0 \\
0      & \gamma_0 < \gamma
\end{cases}
\]
We denote $\gamma_0 = \dim_{\Hd}(S)$ and refer to this as the \emph{Hausdorff dimension} of the set $S$.

\begin{proof}[Proof of Theorem~\ref{t:zhdim}]

Let $\E_0 \in \LP$ be $q_0$-periodic, and suppose $\varepsilon_0 > 0$.  We will construct a sequence $(\E_n)_{n=1}^\infty$ consisting of periodic operators so that $\E_\infty = \lim_n \E_n$ satisfies
\[
\|\E_0 - \E_\infty\|  < \varepsilon_0
\text{ and }
h^\gamma(\sigma(\E_\infty))
=
0
\]
for all $\gamma>0$; evidently, this suffices to show that $\sigma(\E_\infty)$ has Hausdorff dimension zero. Throughout the proof, $\Sigma_{n} = \sigma(\E_n)$ for $1 \leq n \leq \infty$.
\newline

Take $\varepsilon_1 = \varepsilon_0/2$. By Lemma~\ref{l:lp:smallspec}, we may construct a $q_1$-periodic operator $\E_1$ with $\|\E_0 - \E_1 \| < \varepsilon_1$, and for which
\[
\delta_1
:=
\Leb(\Sigma_1)
<
e^{-q_1^{1/2}}.
\]
 Having constructed $\E_{n-1}$ and $\varepsilon_{n-1}$ such that $\delta_{n-1} \eqdef \Leb(\Sigma_{n-1}) < \exp(-q_{n-1}^{1/2})$,  let
\begin{equation}\label{eq:nextepsdef}
\varepsilon_{n}
=
\min\left(\frac{\varepsilon_{n-1}}{2},
\frac{1}{2} n^{-q_{n-1}},
\frac{\delta_{n-1}}{4} \right).
\end{equation}
By Lemma~\ref{l:lp:smallspec}, we may construct a $q_{n} := N_n q_{n-1}$-periodic operator $\E_{n}$ with $\|\E_n - \E_{n-1}\| < \varepsilon_{n}$  such that
\begin{equation} \label{eq:smallspec}
\delta_{n}
\eqdef
\Leb(\Sigma_{n})
<
e^{-q_{n}^{1/2}}.
\end{equation}
Clearly, $\E_\infty = \lim_{n\to\infty}\E_n$ exists and is limit-periodic. From the first condition in \eqref{eq:nextepsdef}, we deduce
\[
\| \E_0 - \E_\infty\|
<
\sum_{j = 1}^\infty \varepsilon_j
\leq
\sum_{j=1}^\infty 2^{-j} \varepsilon_0
=
 \varepsilon_0,
\]
so $\E_\infty \in B_{\varepsilon_0}(\E_0)$. Similarly, using the first and second conditions in \eqref{eq:nextepsdef}, we observe that
\[
\|\E_n - \E_\infty\|
<
n^{-q_n}
\]
for each $n \geq 2$. From this, it is easy to see that $\E_\infty$ is a Gordon operator in the sense of \cite{F2016PAMS}, and hence, $\E_\infty$ has purely continuous spectrum. Thus, it remains to show that the spectrum has Hausdorff dimension zero. The key observation in this direction is that \eqref{eq:nextepsdef} yields
\begin{equation} \label{eq:taildist}
\|\E_n - \E_\infty\|
\leq
\sum_{j = n + 1}^\infty \varepsilon_j
<
\sum_{k = 2}^\infty 2^{-k} \delta_n
=
\delta_n/2
\end{equation}
for all $n \in \Z_+$.

Now, let $\delta > 0$ and $\gamma > 0$ be given. Choose $n \in \Z_+$ for which $2\delta_n < \delta$. Then, by \eqref{eq:taildist}, the $\delta_n/2$-neighborhood of $\Sigma_{n}$ comprises a $\delta$-cover of $\Sigma_{\infty}$; denote this cover by $\mathscr I_n$. By \eqref{eq:smallspec}, we have
\[
\sum_{I \in \mathscr I_n}
\Leb(I)^{\gamma}
\leq
q_n 2^\gamma e^{-\gamma q_n^{1/2}}
\]
Sending $\delta \downarrow 0$ (and hence $n \to \infty$), we have $h^{\gamma}(\Sigma_{\infty}) = 0$. Since this holds for all $\gamma > 0$, $\dim_\Hd(\Sigma_\infty) = 0$, as desired.
\end{proof}

\section{CMV Operators on $\ell^2(\N)$ and Their Schur Functions} \label{sec:OPUC}
We consider a CMV operator on $\ell^2(\N)$ rather than $\ell^2(\Z)$. Specifically, we define
\[
\mathcal L_+
=
\bigoplus_{j \in \Z_{\geq 0}} \Theta(\alpha_{2j},\lambda_{2j}),
\quad
\mathcal M_+
=1\oplus 
\bigoplus_{j \in \Z_{\geq 0}} \Theta(\alpha_{2j+1},\lambda_{2j+1}),
\]
where $\Theta$ is defined as in \eqref{eq:Theta.defn}. To clarify, $1$ is not the $2\times 2$ identity matrix, but rather the identity map from the $0th$ vector entry to itself. We can then define a CMV matrix $\mathcal C=\mathcal L_+\mathcal M_+$ and calculate that
\begin{equation} \label{eq:popuc:def}
\mathcal C
=
\begin{bmatrix}
&\lambda_0\overline{\alpha_0}
&\lambda_1\lambda_0\overline{\alpha_1}\rho_0
&\lambda_1\lambda_0\rho_1\rho_0
&0
&0\\
&\lambda_0\rho_0
&-\lambda_1\lambda_0\overline{\alpha_1}\alpha_0
&-\lambda_1\lambda_0\rho_1\alpha_0
&0
&0\\
&0& \lambda_2 \lambda_1 \overline{\alpha_2}\rho_1 
& -\lambda_2 \lambda_1 \overline{\alpha_2}\alpha_1 
& \lambda_3 \lambda_2\overline{\alpha_3} \rho_2 
& \lambda_3 \lambda_2 \rho_3\rho_2 \\
&0& \lambda_2 \lambda_1 \rho_2\rho_1 
& -\lambda_2 \lambda_1 \rho_2\alpha_1 
& -\lambda_3 \lambda_2 \overline{\alpha_3}\alpha_2 
& -\lambda_3 \lambda_2 \rho_3\alpha_2  \\
&& \ddots & \ddots & \ddots & \ddots
\end{bmatrix},
\end{equation}
Here, we have the same $4\times 2$ block pattern we had in \eqref{eq:pcmv:def}, except that the first block is truncated; in essense, this arises when $\alpha_{-1}$ is set to be $-1$.

When all the $\lambda$'s are equal to $1$, this matrix is of crucial importance in the theory of orthogonal polynomials on the unit circle (OPUC). It arises when we have a probability measure $\mu$ on the unit circle with infinite support. We can then perform a Gram-Schmidt orthogonalization process on the sequence $1,z^{-1},z,z,z^{-2},z^2,\ldots$ with inner product
\[
\left<f(z),g(z)\right>
=
\int_{\partial\D} \overline{f(z)}g(z) \, \dd \mu(z),
\]
which gives us a basis of the space of continuous functions on $\partial \D$.  The CMV matrix \eqref{eq:popuc:def} represents the multiplication by $z$ operator on this basis. Thus, the CMV operator allows us to use spectral theoretic techniques to handle problems in orthogonal polynomials. Please consult \cite{S1},\cite{S2} for an extensive discussion of this point of view.

Of course, we can define limit-periodic $\mathcal C$ the same way we defined limit-periodic $\E$, as the closure of the space of periodic CMV operators under the operator norm topology.

All of our results for the $\ell^2(\Z)$ operator $\E$ also hold for the $\ell^2(\N)$ operator $\mathcal C$, with one exception. The Gordon lemma argument given in \cite{F2016PAMS} and \cite{O12} no longer holds, and therefore we cannot rule out the possibility of point spectrum.

\begin{theorem}
There is a dense $G_\delta$ set of limit-periodic pCMV operators on $\ell^2(\mathbb N)$ which enjoy purely singular spectrum of zero Lebesgue measure. Furthermore, a dense set of such operators possess spectrum of zero Hausdorff dimension.
\end{theorem}

Our results in Section~\ref{sec:dos} also implicitly give us results connecting the Schur function of $\mu$ to the DOS of the associated CMV matrix, $\mathcal C$; in particular, we will see below that (modulo multiplication by $z$ and conjugation), the $\xi$ from Section~\ref{sec:dos} is the Schur function in disguise! Let us now make this more explicit. Each CMV matrix $\mathcal C$ is associated with a Schur function $s$, (i.e.\ an analytic function from the open unit disk to itself). This function is a cousin of the Weyl-Titchmarsh $m$-function for Jacobi and Schr\"odinger operators, in that its limiting behaviour on the boundary gives us information about the spectrum of $\mathcal C$. Letting $\mu$ denote the spectral measure of $\mathcal C$, its associated Schur function $s$ can be defined as 
\begin{equation}\label{eq:Schurdef}
s(z)
=
\frac{1}{z}\frac{F(z)-1}{F(z)+1},
\quad z \in \D,
\end{equation}
where $F(z)$ is the Carath\'eodory function, defined by
\[
F(z)
=
\int \frac{e^{i\theta}+z}{e^{i\theta}-z} \, \dd\mu(\theta),
\quad
z \in \D.
\]

\begin{theorem}\label{t:Schur}
Consider a $q$-periodic CMV matrix $\mathcal C$ on $\ell^2(\N)$ with $q$ even. Let
\[
\Delta(z) = \tr(\Phi(z)),
\quad
\Phi(z) = Z(q,0;z)
\]
be its discriminant and monodromy matrix, respectively and denote by $\dd\nu(\tau)$ its density of states measure, parameterized in the usual way by $\tau\in [0,2\pi)$. For each $j \geq 0$, let $s_j(z)$ be the Schur function corresponding to the $q$-periodic Verblunsky coefficient sequence starting with $\alpha_{j}, \alpha_{j+1},\ldots \alpha_{q-1}, \alpha_0,\ldots,\alpha_{j-1}$. Then,
\[
\dfrac{\dd\nu(\tau)}{\dd\tau}
\geq 
\frac{1}{\pi q} \sum_{j=0}^{\frac{q-2}{2}} \frac{1}{1-|s_{2j}(e^{i\tau})|^2},
\]
whenever $\Delta(e^{i\tau}) \in (-2,2)$.
\end{theorem}

\begin{proof} For $z \in \D\setminus\set{0}$, we may construct a Weyl solution to $\mathcal C \psi = z\psi$. More precisely, with $u = (\mathcal C - z\I)^{-1}\delta_0$ and $v = \mathcal M_+ u$, we have
\begin{equation} \label{eq:CaratheoWeyl}
\begin{bmatrix}
u(n) \\ v(n)
\end{bmatrix}
=
\frac{F(z)+1}{2z} Z(n,0;z) 
\begin{bmatrix}
1 \\ zs(z)
\end{bmatrix},
\quad
n \geq 1,
\end{equation}
by \eqref{eq:Schurdef} and \cite[Proposition~2.3]{DEFHV} (compare \cite[Theorem~3.2.11]{S1}). On the other hand, since $\mathcal C$ is $q$-periodic, the space of initial conditions that lead to square-summable solutions is simply the contracting eigenspace of $\Phi(z)$, so $(1,zs(z))^\top$ spans the contracting eigenspace of $\Phi(z)$ for $z \in \D\setminus\set{0}$. Now, fix $\tau \in [0,2\pi)$ with $\Delta(e^{i\tau}) \in (-2,2)$. By \cite[Section~11.3]{S2}, the real part of the Carath\'eodory function has a boundary value at $w=e^{i\tau}$ with
\[
\Re \, F(w)
>
0,
\]
and hence, \eqref{eq:Schurdef} implies that $s$ has a boundary value with $|s(w)| < 1$. Since $|s| < 1$, the discussion above implies that the M\"obius action of $\Phi(w)$ fixes the point
\[
\frac{1}{ws(w)} \in \C\setminus \overline\D.
\]
On the other hand, $\Phi(w) \in \SU(1,1)$ by assumption, so by \eqref{eq:su11char}, there are complex numbers $p,q$ with
\[
\Phi(w)
=
\begin{bmatrix}
p & q \\ \overline{q} & \overline{p}
\end{bmatrix},
\quad
|p|^2-|q|^2 =1.
\]
Hence, by simple algebra, it follows that $\Phi(w)$ also fixes the point $\overline{ws(w)}$; cf.\ \cite[Theorem~9.3.4]{simszego}. Consequently, since $|s(w)| < 1$, it follows that the fixed point of the M\"obius action of $\Phi(w)$ on $\D$ is precisely
\[
\xi
=
\overline{ws(w)}.
\]

Consequently, by shifting, we get $|\xi_{2j}| = |s_{2j}|$ for each $j$, so, by the proof of Theorem~\ref{t:ids:HSnorm}, we have
\[
\dfrac{d \nu(\tau)}{d\tau}
\geq 
\frac{1}{\pi q}\sum_{j=0}^{\frac{q-2}{2}}\frac{1}{1-\vert \xi_{2j}(e^{i\tau)}\vert^2}
=
\frac{1}{\pi q}\sum_{j=0}^{\frac{q-2}{2}}\frac{1}{1-|s_{2j}(e^{i\tau})|^2},
\]
which proves Theorem~\ref{t:Schur}.
\end{proof}

\appendix
\section{A Discrete-Time RAGE Theorem} \label{sec:RAGE}
We will present a discrete-time variant of the RAGE theorem. A partial version has been given in \cite{HJS}. The  proof for singular continuous and absolutely continuous spectrum that we give follows the arguments in the continuous-time (self-adjoint) setting in \cite{Kirsch}.

Our proof will rely on Wiener's theorem on the circle.

\begin{theorem}[Unitary Wiener's Theorem {\cite[Theorem~12.4.7]{S2}}]\label{t:Wiener}
Let $\mu$ be a probability measure on $\partial \D$ and let $c_n$ be its moments,
\[
c_n
=
\int_{\partial \D} z^n \, \dd\mu(z).
\]
Then
\[
\lim_{N\to\infty} \frac{1}{2N+1}\sum_{n=-N}^N \vert c_n\vert^2
=
\sum_{z\in\partial \D}\vert \mu(\{z\})\vert^2.
\]
In particular, $\mu_{\pp}=0$ if and only if 
\[
\lim_{N \to \infty} \frac{1}{2N+1}\sum_{n=-N}^N\vert c_n\vert^2
=
0.
\]
\end{theorem}

\begin{theorem}[Discrete-time RAGE theorem]\label{t:rage}
Let $U$ be a unitary operator on $\ell^2(\Z)$, $\psi \in \ell^2(\Z)$, and $\mu = \mu_{\psi,U}$ the associated spectral measure.\\
\begin{enumerate}
\item[{\rm(a)}] If $\mu = \mu_\pp$, then for every $\epsilon > 0$ there is a $J\in\Z_+$ such that
\begin{equation}\label{eq:RAGEpp}
\limsup_{n\to\infty}\sum_{\vert j\vert \geq J}\left\vert \left
<
\delta_j, U^n \psi\right>\right\vert^2<\epsilon
\end{equation}
\item[{\rm(b)}]  If $\mu = \mu_{\mathrm c}$ {\rm(}i.e.\ $\mu_\pp = 0${\rm)} then for any $J\in \Z_+$ we have
\begin{equation}\label{eq:RAGEc}
\lim_{N\to\infty} \frac{1}{2N+1}\sum_{n=-N}^N\sum_{j=-J}^J \left\vert \left< \delta_j, U^n \psi\right>\right\vert^2
=
0.
\end{equation}
\item[{\rm(c)}] If $\mu = \mu_\ac$ then for any $J\in\Z_+$,
\begin{equation}\label{eq:RAGEac}
\lim_{n\to\infty}  \sum_{j=-J}^J\left\vert \left< \delta_j, U^n \psi\right>\right\vert^2
=
0.
\end{equation}

\end{enumerate}

\end{theorem}

\begin{proof}\ \\
\begin{enumerate}[(a)]
\item If $\psi$ is a finite linear combination of eigenfunctions of $U$, say
\[
\psi
=
\sum_{\ell=1}^k \psi_\ell,
\]
with $U\psi_\ell = z_\ell \psi_\ell$, then \eqref{eq:RAGEpp} obviously holds. Let us consider now an arbitrary $\psi$ with $\mu = \mu_\pp$. Since finite linear combinations of eigenfunction are dense in the pure point subspace, there exists $\psi'$, a finite linear combination of eigenfunctions such that $\Vert \psi-\psi'\Vert^2 < \epsilon/4$ (using the $\ell^2$ norm). Thus, we may choose $J \in \Z_+$ such that
\[
\limsup_{n\to\infty}\sum_{\vert j\vert \geq J}\left\vert \left< \delta_j, U^n \psi'\right>\right\vert^2
<
\epsilon/4.
\] 
We then have
\[
\sum_{\vert j\vert \geq J}\left\vert \left< \delta_j, U^n \psi\right>\right\vert^2
\leq 
2\sum_{\vert j\vert \geq J}\left\vert \left< \delta_j, U^n \psi'\right>\right\vert^2
+2\sum_{\vert j\vert \geq J}\left\vert \left< \delta_j, U^n (\psi-\psi')\right>\right\vert^2.
\]
Using this, we get
\begin{align*}
\limsup_{n\to\infty} \sum_{\vert j\vert \geq J}\left\vert \left< \delta_j, U^n \psi\right>\right\vert^2
& \leq
\frac{\epsilon}{2} + 2\limsup_{n\to\infty}\Vert U^n (\psi-\psi')\Vert^2 \\
& = 
\frac{\epsilon}{2}+ 2\Vert \psi-\psi'\Vert^2 \\
& <
\epsilon.
\end{align*} 
\item  Fix $j \in \Z$, and let $\mu_j$ denote the spectral measure corresponding to the pair $(\delta_j,\psi)$. Since $\mu_j \ll \mu$, it follows that $\mu_j$ has no pure points, and hence
\[
\lim_{N \to \infty} \frac{1}{2N+1} \sum_{|n| \le N} |\langle \delta_j,U^n \psi\rangle|^2
=
\lim_{N \to \infty} \frac{1}{2N+1}  \sum_{|n| \le N} \left|\widehat \mu_j(n) \right|^2
=
0
\]
by Wiener's theorem. Thus, for each fixed $J \in \Z_+$, one has
\[
\lim_{N \to \infty} \frac{1}{2N+1} \sum_{|j| \le J} \sum_{|n| \le N} |\langle \delta_j,U^n \psi\rangle|^2
=
\sum_{|j| \le J} \lim_{N \to \infty} \frac{1}{2N+1}  \sum_{|n| \le N} |\langle \delta_j,U^n \psi\rangle|^2
=
0,
\]
which gives \eqref{eq:RAGEc}.
\medskip

\item Let $\mu_j$ be as in part (b). Since $\mu_j$ is absolutely continuous with respect to $\mu$, it is absolutely continuous with respect to Lebesgue measure. Consequently, we can define an $L^1(\partial \D)$ function $h_j$ with the property that
\[
\left< \delta_j, U^n\psi\right>
=
\int_{\partial \D} \! z^n \,  \dd\mu_j(z)
=
\int_{\partial \D} z^n h_j(z) \, \dd z.
\] 
The terminal expression is the $n$th Fourier coefficient of the $L^1$ function $h_j(z)$, which goes to zero as $|n| \to \infty$ by the Riemann--Lebesgue Lemma. Thus, \eqref{eq:RAGEac} follows.
\end{enumerate}
\end{proof}

\end{document}